\documentclass[11pt,a4paper,draft]{article}
\usepackage{mathrsfs}

\usepackage[leqno]{amsmath}
\usepackage{amssymb,amsthm,upref,amscd}
\usepackage[T1]{fontenc}
\numberwithin{equation}{section}
\input scrload.tex

\theoremstyle{plain}
\newtheorem{Thm}{Theorem}[section]
\newtheorem{Lem}[Thm]{Lemma}

\newtheorem*{Thm*}{Theorem}
\theoremstyle{definition}

\newtheorem{Rem}[Thm]{Remark}

\newcommand{\C}{\mathbb{C}}

\newcommand{\R}{\mathbb{R}}

\let\nhatoksa=\theenumi
\let\nhatoksb=\labelenumi
\let\nhatoksc=\theenumii
\let\nhatoksd=\labelenumii
\newlength{\nhalengtha}
\newlength{\nhalengthb}
\newlength{\nhalengthc}
\newcommand{\resetenum}{
\let\theenumi=\nhatoksa
\let\labelenumi=\nhatoksb
\let\theenumii=\nhatoksc
\let\labelenumii=\nhatoksd
\setlength{\leftmargini}{\nhalengtha}
\setlength{\leftmarginii}{\nhalengthb}
\setlength{\labelwidth}{\nhalengthc}
}

\def\ov{\overline}

\def\pa {\partial}

\def\al {\alpha}
\def\bt {\beta}
\def\de {\delta}
\def\Ga {\Gamma}
\def\ga {\gamma}
\def\lm {\lambda}

\def\om {\omega}

\def\vr {\varepsilon}
\def\va {\varphi}

\def\fa{\mathbf{A}}
\newcommand{\bkt}[1]{\left(#1\right)}

\newcommand{\hbt}[1]{\left\{#1\right\}}
\newcommand{\norm}[1]{\left\|#1\right\|}
\newcommand{\inp}[2]{\left\langle#1,#2\right\rangle}
\newcommand{\jdz}[1]{\left|#1\right|}

\title{On semi-classical limits of ground states of a
nonlinear Maxwell-Dirac system}

\author{Yanheng Ding \,
and \, Tian Xu \and {\small Institute
of Mathematics, AMSS, Chinese Academy of Sciences,} \\
{\small  100190 Beijing, China} }

\date{}

\begin{document}
\maketitle

\begin{abstract}
We study the semi-classical ground states of the nonlinear
Maxwell-Dirac system:
\[
\left\{
\begin{aligned}
&\al\cdot\big(i\hbar\nabla+ q(x)\fa(x)\big) w-a\bt w -\omega w - q(x)\phi(x) w
= P(x)g(\jdz{w}) w\\
&-\Delta\phi=q(x)\jdz{w}^2\\
&-\Delta{A_k}=q(x)(\alpha_k  w)\cdot \bar w\ \ \ \ k=1,2,3
\end{aligned}\right.
\]
for $x\in\R^3$, where $\fa$ is the magnetic field, $\phi$ is the
electron field and $q$ describes the changing pointwise charge distribution.
We develop a variational method to establish the
existence of least energy solutions for $\hbar$ small. We also
describe the concentration behavior of the solutions as $\hbar\to 0$.

\vspace{.5cm}


\vspace{.5cm}


\end{abstract}

\section{Introduction and main result}

The Maxwell-Dirac system, which has been widely considered in
literature (see \cite{Abenda}, \cite{Sere2}, \cite{Glassey},
\cite{Lisi}, \cite{Psarelli}, \cite{Schwabl}, \cite{Thaller} etc. and
references therein), is fundamental in the relativistic description
of spin $1/2$ particles. It represents the time-evolution of fast
(relativistic) electrons and positrons within external and
self-consistent generated electromagnetic field. The system
can be written as follows:
\begin{equation}\label{M-D0}
\left\{
\begin{aligned}
&i\hbar\partial_t \psi+\alpha\cdot\bkt{ic\hbar\nabla+q\mathbf{A}}\psi-q\phi\psi-mc^2\bt \psi=0\\
&\partial_t\phi+c\sum_{k=1}^3\partial_k A_k=0\,,
\ \ \ \ \partial_t^2\phi-\Delta\phi=\frac{4\pi}{c}q\jdz{\psi}^2\\
&\partial_t^2A_k-\Delta A_k=\frac{4\pi}{c}q(\alpha_k \psi)\bar \psi\ \ \ \ k=1,2,3
\end{aligned}\ \ \ \ \mathrm{in}\ \R\times\R^3 \right.
\end{equation}
where $\psi(t,x)\in\mathbb{C}^4$, $c$ is the speed of light, $q$ is
the charge of the particle, $m>0$ is the mass of the electron,
$\hbar$ is the Planck's constant, and $u\bar{v}$ denotes the inner
product of $u, v\in\mathbb{C}^4$. Furthermore, $\alpha_1$,
$\alpha_2$, $\alpha_3$ and $\beta$ are $4\times4$ complex matrices:
\[
\beta=\left(
\begin{array}{cc}
I&0\\
0&-I
\end{array}\right),\ \ \ \alpha_k=\left(
\begin{array}{cc}
0&\sigma_k\\
\sigma_k&0
\end{array}\right),\ \ k=1,2,3,
\]
with
\[
\sigma_1=\left(
\begin{array}{cc}
0&1\\
1&0
\end{array}\right),\ \ \ \sigma_2=\left(
\begin{array}{cc}
0&-i\\
i&0
\end{array}\right),\ \ \ \sigma_3=\left(
\begin{array}{cc}
1&0\\
0&-1
\end{array}\right),
\]
$\mathbf{A}=(A_1,A_2,A_3):\mathbb{R}\times\mathbb{R}^3\to\mathbb{R}^3$,
$\phi:\mathbb{R}\times\mathbb{R}^3\to\mathbb{R}$, and we have used
$\al=(\al_1, \al_2, \al_3)$,
$\al\cdot\nabla=\sum_{k=1}^3\al_k\partial_k$, and
$\al\cdot\mathbf{V}=\sum_{k=1}^3\al_k V_k$ for any vector
$\mathbf{V}\in\mathbb{C}^3$.

The above system has been studied for a long time and results are
available concerning the Cauchy problem (see \cite{Chadam1},
\cite{Chadam2}, \cite{Flato},
 \cite{Georgiev}, \cite{Gross}, \cite{Sparber} etc. and
references therein). The first result on the local existence and
uniqueness of solutions of \eqref{M-D0} was obtained by L. Gross
in \cite{Gross}. For later d\'eveloppements, we mention, e.g.,
that Sparber and Markowich \cite{Sparber} studied the
existence and asymptotic description of
the solution of Cauchy problem for Maxwell-Dirac system as
$\hbar\to0$, and obtained the asymptotic approximation
up to order $O(\sqrt{\hbar})$.

In this paper, we are interested in finding stationary waves of \eqref{M-D0}
which have the form
\[\left\{
\aligned
&\psi(t,x)=w(x)e^{i\theta t/\hbar},\ \ \theta\in\mathbb{R}, \ \ w:\R^3\to \C^4,\\
&\fa=\fa(x), \quad  \phi=\phi(x) \quad \text{in }  \R^3.
\endaligned \right.
\]
For notation convenience, one shall denote $A_0=\phi$. If $(\psi, \fa, A_0)$ is
a stationary solution of \eqref{M-D0}, then $(w, \fa, A_0)$ is a solution of
\begin{equation}\label{D1}
\left\{
\begin{aligned}
&\al\cdot\bkt{i\hbar\nabla+Q\fa} w-a\beta w-\omega w-QA_0 w=0,\\
&-\Delta A_k=4\pi Q(\alpha_k w) \bar{w},\ \ \ \ k=0,1,2,3,
\end{aligned}\right.
\end{equation}
where $a=mc>0$, $\omega\in\mathbb{R}$, $Q=q/c$ and $\alpha_0:=I$.

The existence of stationary solution of the system has been an open
problem for a long time, see \cite{Grandy}. Using variational
methods, Esteban, Georgiev and S\'{e}r\'{e} \cite{Sere1} proved the
existence of regular solutions of the form $\psi(t,x)=w(x)e^{i\om t}$ with $\om\in(0,a)$,
leaving open the question of existence of solutions for
$\omega\leq0$. On the other hand, in \cite{Lisi}, Garrett Lisi gave
numerical evidence of the existence of bounded states for
$\omega\in(-a,a)$ by using an axially symmetric ansatz. After that,
Abenda in \cite{Abenda} obtained the existence result of solitary
wave solutions for $\omega\in(-a,a)$.

We emphasize that the works mentioned above mainly concerned with
the autonomous system with null self-coupling. Besides, limited work has
been done in the semi-classical approximation.
For small $\hbar$, the solitary waves are
referred to as semi-classical states. To describe the transition
from quantum to classical mechanics, the existence of solutions
$w_\hbar$, $\hbar$ small, possesses an important physical interest.
The idea to consider a nonlinear self-coupling,
in Quantum electrodynamics,
gives the description of models of self-interacting spinor fields (
see \cite{FLR}, \cite{FFK}, \cite{Iva} etc. and references therein).
Due to the special physical importance, in the present paper,
we are devoted to the existence and
concentration phenomenon of stationary
semi-classical solutions to the system with
\begin{itemize}
\item   the varying pointwise charge distribution $Q(x)$ including the
constant $q$ as a special one;
\item   general subcritical self-coupling nonlinearity.
\end{itemize}
More precisely, we consider the system, writing $\vr =\hbar$,
\begin{equation}\label{D2}
\left\{
\begin{aligned}
&\al\cdot\bkt{i\vr \nabla+Q(x)\fa} w-a\beta w-\omega w-Q(x)A_0 w = P(x)g(\jdz{w}) w,\\
&-\Delta A_k=4\pi Q(x)(\al_k w) \bar{w}\ \ \ \ k=0,1,2,3.
\end{aligned}\right.
\end{equation}
Writing $G(\jdz{w}):=\int^{\jdz{w}}_0g(s)sds$, we make the following
hypotheses:
\begin{itemize}
\item[$(g_1)$] {\it $g(0)=0$, $g\in C^1(0,\infty)$, $g'(s)>0$ for $s>0$, and
there exist $p\in (2, 3)$, $c_1>0$ such that $g(s)\leq
c_1(1+s^{p-2})$ for $s\geq 0$}\ ;
\item[$(g_2)$]{\it  there exist $\sigma>2$, $\theta>2$ and $c_0>0$ such that
$c_0s^\sigma\leq G(s)\leq \frac{1}{\theta}g(s)s^2$ for all $s>0$}\ .
\end{itemize}
A typical example is the power function $g(s)=s^{\sigma-2}$. For
describing the charge distribution and external fields
we always assume that $Q(x)$ and $P(x)$ verify, respectively

\begin{itemize}
\item[$(Q_0)$] $Q\in C^{0,1}(\mathbb{R}^3)\cap L^\infty(\mathbb{R}^3)$
 with $Q(x)\geq0$ a.e.
on $\mathbb{R}^3$;
\item[$(P_0)$] $P\in C^{0,1}(\mathbb{R}^3)\cap L^\infty(\mathbb{R}^3)$ 
 with $\inf P>0$ and $\limsup\limits_{\jdz{x}\to\infty}P(x)<\max P(x)$.
\end{itemize}
For showing the concentration phenomena,  we set $m:=\max_{x\in\mathbb{R}^3}P(x)$ and
$$
\mathscr{P}:=\{x\in\mathbb{R}^3:P(x)=m\}.
$$

Our result reads as

\begin{Thm}\label{main theorem}
Assume that $\omega\in(-a,a)$, $(g_1)$-$(g_2)$, $(Q_0)$ and $(P_0)$ are
satisfied. Then for all $\vr >0$ small,
\begin{itemize}
\item[$(i)$] The system (\ref{D2}) has at least one least energy solution $w_\vr \in W^{1,q}$
 for all $q\geq2$. In addition, if $P,\ Q\in C^{1,1}(\mathbb{R}^3)$ the solutions will be in $C^1$ class.
\item[$(ii)$] The set of all least energy solutions is compact in $W^{1,q}$ for all $q\geq2$;
\item[$(iii)$] There is a maximum point $x_\vr $ of
$\jdz{w_\vr }$ with
$\lim_{\vr \to0}\mathrm{dist}(x_\vr ,\mathscr{P})=0$ such that $u_\vr (x):=w_\vr (\vr  x+x_\vr )$
converges uniformly to a least energy solution of (the  limit
equation)
\begin{equation}\label{the limit problem}
i\alpha\cdot\nabla u-a\beta u-\omega u=mg(\jdz{u})u.
\end{equation}
\item[$(iv)$]
$\jdz{w_\vr (x)}\leq C\exp{\bkt{-\frac{c}{\vr }\jdz{x-x_\vr }}}$
for some $C,c>0$.
\end{itemize}
\end{Thm}

It is standard that (\ref{D2}) is equivalent to, letting $u(x)=
w(\vr x)$,
\begin{equation}\label{D3}
\left\{
\begin{aligned}
&\al\cdot\bkt{i\nabla+Q_\vr \fa_{\vr }}u-a\beta u-\omega u-Q_\vr  A_{\vr ,0}u=P_\vr \, g(\jdz{u})u,\\
&-\Delta A_{\vr ,k}=\vr ^24\pi Q_\vr  J_k\ \ \ \ k=0,1,2,3,
\end{aligned}\right.
\end{equation}
where $Q_\vr (x)=Q(\vr  x)$, $P_\vr (x)=P(\vr x)$,
$\fa_{\vr}(x)=\fa(\vr x)$, $A_{\vr,k}(x)=A_{k}(\vr x)$, $k=0,1,2,3$,
and
$$
J_k=(\alpha_ku)\bar{u}\ \ \mathrm{for}\ k=0,1,2,3.
$$
In fact, using variational methods, we are going to focus on studying the
semiclassical solutions that are obtained as critical
points of an energy functional $\Phi_\vr $ associated to the
equivalent problem \eqref{D3}.

There have been a large number of works on existence
and concentration phenomenon of semi-classical
states of nonlinear Schr\"odinger-Poisson
systems arising in the \textit{non-relativistic} quantum mechanics,
see, for example, \cite{Ambrosetti1, Ambrosetti2, Azzollini}
and their references.
It is quite natural to ask if certain similar
results can be obtain for nonlinear Maxwell-Dirac
systems arising in the relativistic quantum mechanics.
Mathematically, the two systems possess different
variational structures, the Mountain-Pass and the Linking structures
respectively. The problems in Maxwell-Dirac systems are difficult because
they are strongly indefinite in the sense that both the negative and positive
parts of the spectrum of Dirac operator are unbounded
and consist of essential spectrums.
As far as the authors known there have been no results
on the existence and concentration phenomena of semiclassical solutions to
nonlinear Maxwell-Dirac systems.

Very recently, one of the authors, jointly with co-authors,
developed an argument to obtain some results on existence
and concentration of semi-classical solutions for
nonlinear Dirac equations but not for Maxwell-Dirac
system, see \cite{Ding2010, Ding2012, Ding Ruf}.
Compared with
the papers, difficulty arises in the Maxwell-Dirac system because of
the presence of nonlocal terms $A_{\vr ,k}$, $k=0,1,2,3$. In order
to overcome this obstacle, we develop a cut-off arguments. Roughly
speaking, an accurate uniform boundness estimates on
$(C)_c$ (Cerami) sequences of the associate energy functional $\Phi_\vr $
enables us to introduce a new functional $\widetilde{\Phi}_\vr $ by
virtue of the cut-off technique so that $\widetilde{\Phi}_\vr $ has
the same least energy solutions as ${\Phi}_\vr $ and can be dealt
with more easily, in particular, the influence of these nonlocal
terms can be reduced as $\vr \to0$. In addition, for obtaining the
exponential decay, since the Kato's inequality seems not work well
in the present situation, we handle, instead of considering $\Delta
|u|$ as in \cite{Ding2010}, the square of $|u|$, that is $\Delta
|u|^2$, with the help of identity \eqref{identity}, and then
describe the decay at infinity in a subtle way.


\section{The variational framework}

\subsection{The functional setting and notations}

In this subsection we discuss the variational setting for the
equivalent system \eqref{D3}. Throughout the paper we assume
$0\in\mathscr{P}$ without loss of generality, and the conditions
$(g_1)$-$(g_2)$, $(P_0)$ and $(Q_0)$ are satisfied.

In the sequel, by $|\cdot|_q$ we denote the usual $L^q$-norm, and
$(\cdot,\cdot)_2$ the usual $L^2$-inner product. Let
$H_0=i\alpha\cdot\nabla-a\beta$ denote the self-adjoint operator on
$L^2\equiv L^2(\mathbb{R}^3,\mathbb{C}^4)$ with domain
$\mathcal{D}(H_0)=H^1\equiv H^1(\mathbb{R}^3,\mathbb{C}^4)$. It is
well known that $\sigma(H_0)=\sigma_c(H_0)=\mathbb{R}\setminus(-a,a)$
where $\sigma(\cdot)$ and $\sigma_c(\cdot)$ denote the spectrum and
the continuous spectrum. Thus the space $L^2$ possesses the
orthogonal decomposition:
\begin{equation}\label{l2dec}
L^2=L^+\oplus L^-,\ \ \ \ u=u^++u^-
\end{equation}
so that $H_0$ is positive definite (resp. negative definite) in
$L^+$ (resp. $L^-$). Let $E:=\mathcal{D}(\jdz{H_0}^{1/2})=H^{1/2}$
be equipped with the inner product
$$
\inp{u}{v}=\Re(\jdz{H_0}^{1/2}u,\jdz{H_0}^{1/2}v)_2
$$
and the induced norm $\norm{u}=\inp{u}{u}^{1/2}$, where $\jdz{H_0}$
and $\jdz{H_0}^{1/2}$ denote respectively the absolute value of
$H_0$ and the square root of $|H_0|$. Since
$\sigma(H_0)=\mathbb{R}\setminus(-a,a)$, one has
\begin{equation}\label{l2ineq}
a|u|_2^2\leq\norm{u}^2 \quad  \mathrm{for\ all\ }u\in E.
\end{equation}
Note that this norm is equivalent to the usual $H^{1/2}$-norm, hence
$E$ embeds continuously into $L^q$ for all $q\in[2,3]$ and compactly
into $L_{loc}^q$ for all $q\in[1,3)$. It is clear that $E$ possesses
the following decomposition
\begin{equation}\label{Edec}
E=E^+\oplus E^-\ \ \mathrm{with\ \ }E^{\pm}=E\cap L^{\pm},
\end{equation}
orthogonal with respect to both $(\cdot,\cdot)_2$ and
$\inp{\cdot}{\cdot}$ inner products. This decomposition induces also
a natural decomposition of $L^p$, hence there is $d_p>0$ such that
\begin{equation}\label{lpdec}
d_p\jdz{u^\pm}_p^p\leq\jdz{u}_p^p\ \ \mathrm{for\ all\ }u\in E.
\end{equation}

Let $\mathcal{D}^{1,2}\equiv
\mathcal{D}^{1,2}(\mathbb{R}^3,\mathbb{R})$ be the completion of
$C_c^\infty(\mathbb{R}^3,\mathbb{R})$ with respect the Dirichlet
norm
$$
\norm{u}_{\mathcal{D}}^2=\int \jdz{\nabla u}^2dx.
$$
Then (\ref{D3}) can be reduced to a single equation with a non-local
term. Actually, since $Q$ is bounded and $u\in L^q$ for all
$q\in[2,3]$, one has $Q_\vr \jdz{u}^2\in L^{6/5}$ for all $u\in E$,
and there holds, for all $v\in\mathcal{D}^{1,2}$,
\begin{equation}\label{R1}
\begin{aligned}
\jdz{\int Q_\vr (x)J_k\cdot vdx}&\leq\bkt{\int \jdz{Q_\vr (x)\jdz{u}^2}^{6/5}dx}^{5/6}\bkt{\int \jdz{v}^6}^{1/6}\\
 &\leq S^{-1/2}\jdz{Q_\vr \jdz{u}^2}_{6/5}\norm{v}_{\mathcal{D}},
\end{aligned}
\end{equation}
where $S$ is the Sobolev embedding constant: $S|u|^2_6\leq
\|u\|^2_{\mathcal{D}}$ for all $u\in \mathcal{D}^{1,2}$. Hence there
exists a unique $A_{\vr ,u}^k\in\mathcal{D}^{1,2}$ for $k=0,1,2,3$
such that
\begin{equation}\label{solution of poisson}
\int \nabla A_{\vr ,u}^k\nabla v dx=\vr ^24\pi\int Q_\vr (x)J_k
v\,dx
\end{equation}
for all $v\in\mathcal{D}^{1,2}$. It follows that $A_{\vr ,u}^k$
satisfies the Poisson equation
$$
-\Delta A_{\vr ,u}^k=\vr ^24\pi Q_\vr (x)J_k
$$
and there holds
\begin{equation}\label{juan ji}
A_{\vr ,u}^k(x)=\vr ^2\int
\frac{Q_\vr (y)J_k(y)}{\jdz{x-y}}dy=\frac{\vr ^2}{\jdz{x}}*(Q_\vr  J_k).
\end{equation}
Substituting $A_{\vr ,u}^k$, $k=0,1,2,3$, in (\ref{D3}), we are led
to the equation
\begin{equation}\label{D22}
H_0 u-\omega u-Q_\vr (x)A_{\vr ,u}^0u
+\sum_{k=1}^3Q_\vr (x)\alpha_kA_{\vr ,u}^ku=P_\vr (x)g(\jdz{u})u.
\end{equation}

On $E$ we define the functional
$$
\Phi_\vr (u)=\frac{1}{2}\bkt{\|u^+\|^2-\|u^-\|^2
-\omega\jdz{u}_2^2}-\Gamma_\vr (u)-\Psi_\vr (u)
$$
for $u=u^++u^-$, where
$$
\Gamma_\vr (u)=\frac{1}{4}\int Q_\vr (x)A_{\vr ,u}^0(x)J_0dx
-\frac{1}{4}\sum_{k=1}^3\int Q_\vr (x)A_{\vr ,u}^kJ_k\,dx
$$
and
$$
\Psi_\vr (u)=\int P_\vr (x)G(\jdz{u})dx.
$$

\subsection{Technical results}

In this subsection, we shall introduce some lemmas that related
to the functional $\Phi_\vr$.

\begin{Lem}\label{variation}
Under the hypotheses $(g_1)$-$(g_2)$,  one has
$\Phi_\vr \in C^2(E,\R)$ and any critical point
of $\Phi_\vr $ is a solution of \eqref{D3}.
\end{Lem}

\begin{proof}
Clearly, $\Psi_\vr\in C^2(E,\R)$. It remains to check that
$\Ga_\vr\in C^2(E,\R)$. It suffices to show that, for any $u,v\in
E$,
\begin{equation}\label{estimates of Gamma-eps0}
\jdz{\Ga_\vr (u)}\leq\vr ^2C_1\jdz{Q}_\infty^2\norm{u}^4,
\end{equation}
\begin{equation}\label{estimates of Gamma-eps1}
\jdz{\Ga_\vr '(u)v}\leq\vr
^2C_2\jdz{Q}_\infty^2\norm{u}^3\norm{v},
\end{equation}
\begin{equation}\label{estimates of Gamma-eps2}
\jdz{\Gamma_\vr ''(u)[v,v]}\leq\vr
^2C_3\jdz{Q}_\infty^2\norm{u}^2\norm{v}^2.
\end{equation}

Observe that one has, by \eqref{R1} and \eqref{solution of poisson}
with $v= A_{\vr ,u}^k$,
\begin{equation}\label{Ajsinequ}
|A_{\vr ,u}^k|_6\leq S^{-1/2}\|A_{\vr ,u}^k\|_{\mathcal{D}}
\leq \vr ^2C_1\jdz{Q}_\infty\norm{u}^2.
\end{equation}
This, together with the H\"older inequality (with $r=6, r'=6/5$),
implies \eqref{estimates of Gamma-eps0}. Note that $\Ga'_\vr(u)v
= \frac{d}{dt}\Ga_\vr(u+tv)\big|_{t=0}$, so
\begin{equation}\label{R2}
\aligned \Ga'_{\vr}(u)v =
&\frac{\vr^2}{2}\iint\frac{Q_\vr(x)Q_\vr(y)}{|x-y|}\Big(J_0(x)\Re[\al_0
u\ov v(y)]+J_0(y)\Re[\al_0u\ov v(x)]\\
& \, -\sum^3_{k=1}\big(J_k(x)\Re[\al_ku\ov
v(y)]+J_k(y)\Re[\al_ku\ov v(x)]\big)\Big)dydx\\
= &\, \int\Big( Q_\vr A^0_{\vr,u}\Re[\al_0u\ov v]-\sum^3_{k=1}
Q_\vr A^k_{\vr,u}\Re[\al_ku\ov v] \Big) dx
\endaligned
\end{equation}
which, together with the H\"older's inequality and \eqref{Ajsinequ},
shows \eqref{estimates of Gamma-eps1}. Similarly,
\[\aligned
\Ga^{''}_{\vr}(u)[v,v] = &\,\int Q_\vr\Big(
A^0_{\vr,u}J^v_k-\sum^3_{k=1}A^k_{\vr,u}J^v_k\Big) dx \\
&\,\qquad
+2\,\vr^2\iint\frac{Q_\vr(x)Q_\vr(y)}{|x-y|}\Big[\big(\Re[\al_0u\ov
v(x)]\big)\big(\Re[\al_0u\ov v(y)]\big)\\
&\,\qquad -\sum^3_{k=1}\big(\Re[\al_ku\ov
v(x)]\big)\big(\Re[\al_ku\ov v(y)]\big) \Big] dx dy
\endaligned
\]
where $J^u_k=\al_ku\ov u$ and $J^v_k=\al_kv\ov v$, and one gets
\eqref{estimates of Gamma-eps2}.

Now it is a standard to verify that critical points of $\Phi_\vr$
are solutions of \eqref{D3}.
\end{proof}

\medskip

We show further the following:
\begin{Lem}\label{Gamma-eps nonnegative}
For every $\vr >0$, $\Ga_\vr $ is nonnegative and
weakly sequentially lower semi-continuous.
\end{Lem}

\begin{proof}
Firstly, let us recall some technical results in \cite{Sere1}:
For any $\xi=(\xi_0,\xi_1,\xi_2,\xi_3)\in\R^4$ and $u\in\C^4$, we have
\begin{equation}\label{XX1}
\aligned
&\bigg|
\xi_0(\bt u,u)+\sum_{k=1}^3\xi_k(\al_k u,u)
\bigg|^2    \\
=&\, \bigg|
\bigg( \bt u, \Big[ \xi_0 + \sum_{k=1}^3\xi_k\pi_k \Big] u \bigg)
\bigg|^2    \\
\leq&\, |\bt u|^2_{\C^4} \bigg(
u, \Big( \xi_0-\sum_{k=1}^3\xi_k\pi_k \Big)
\Big( \xi_0+\sum_{k=1}^3\xi_k\pi_k \Big) u
\bigg)   \\
=&\, |\xi|^2_{\R^4} |u|^4_{\C^4}.
\endaligned
\end{equation}
Here, we have used the formulas $(u,v)=u\bar v$ for all
$u,v\in\C^4$, $\pi_k=\bt\cdot\al_k$,
$\bt^*=\bt$, $\pi_k^*=-\pi_k$ and
$\pi_i\pi_j+\pi_j\pi_j=-2\de_{ij}$, $1\leq i,j\leq3$. As a consequence, we find
\begin{equation}\label{XX2}
(\bt u, u)^2+\sum_{k=1}^3(\bt u, \pi_k u)^2 \leq |u|^4_{\C^4}.
\end{equation}
So, taking $u(x)\in E$, $x,y\in\R^3$, $\xi_0=\pm(\bt u, u)(y)$, $\xi_k=(\bt u, \pi_k u)(y)$, we get from \eqref{XX1} and \eqref{XX2}
that
\begin{equation}\label{XX3}
\aligned
&\pm(\bt u, u)(y)(\bt u, u)(x) +
\sum_{k=1}^3(\bt u, \pi_k u)(y)(\bt u, \pi_k u)(x)   \\
=&\,\pm(\bt u, u)(y)(\bt u, u)(x) +
\sum_{k=1}^3(\al_k u, u)(y)(\al_k u, u)(x)   \\
\leq&\, |\xi|_{\R^4} |u(x)|_{\C^4}^2\leq |u(y)|_{\C^4}^2|u(x)|_{\C^4}^2.
\endaligned
\end{equation}

It is not difficult to see from \eqref{XX3} that
\begin{equation}\label{sere ineq}
J_0(x)J_0(y)-\sum_{k=1}^3J_k(x)J_k(y)\geq0.
\end{equation}
And hence (see \eqref{juan ji})
$$
\Ga_\vr (u)=\frac{\vr ^2 }{4}\iint\frac{Q_\vr (x)Q_\vr (y)\bkt{J_0(x)J_0(y)
-\sum_{k=1}^3J_k(x)J_k(y)}}{\jdz{x-y}}dxdy\geq0.
$$
And if $u_n\rightharpoonup u$ in $E$, then $u_n\rightarrow u$ a.e..
Therefore (\ref{sere ineq}) and Fatou's lemma yield
$$
\Ga_\vr (u)\leq\liminf_{n\to\infty}\Ga_\vr (u_n)
$$
as claimed.
\end{proof}

Set, for $r>0$, $B_r=\{u\in E:\norm{u}\leq r\}$, and for $e\in E^+$
$$E_e:=E^-\oplus\mathbb{R}^+e$$
with $\R^+=[0,+\infty)$. In virtue of the assumptions
$(g_1)$-$(g_2)$, for any $\de>0$, there exist $r_\de>0, c_\de>0$ and $c_\de'>0$ such that
\begin{equation}\label{g estimates}
\left\{
\begin{aligned}
&g(s)<\de\ \ \mathrm{for\ all\ }0\leq s \leq r_\de;\\
&G(s)\geq c_\de\,s^\theta-\de\,s^2\ \ \mathrm{for\ all\ } s\geq 0;\\
&G(s)\leq \de\,s^2+c_\de'\,s^p\ \ \mathrm{for\ all\ } s\geq 0
\end{aligned}\right.
\end{equation}
and
\begin{equation}\label{R9}
\widehat{G}(s):=\frac12g(s)s^2-G(s)\geq\frac{\theta-2}{2\theta}g(s)s^2
\geq\frac{\theta-2}{2}G(s)\geq c_\theta s^\sigma
\end{equation}
for all $s\geq 0$, where $c_\theta=c_0(\theta-2)/2$.

\begin{Lem}\label{max Phi-eps<C}
For all $\vr \in(0,1]$, $\Phi_\vr $ possess the linking structure:
\begin{itemize}
\item[$1)$] There are $r>0$ and $\tau>0$, both independent of $\vr $, such
that $\Phi_\vr |_{B_r^+}\geq0$ and $\Phi_\vr |_{S_r^+}\geq\tau$, where
\[
B_r^+=B_r\cap E^+=\{u\in E^+:\|u\|\leq r\},
\]
\[
S_r^+=\pa B_r^+=\{u\in E^+:\|u\|= r\}.
\]
\item[$2)$] For any $e\in E^+\setminus\{0\}$, there exist $R=R_e>0$ and
$C=C_e>0$, both independent of $\vr $, such that, for all $\vr >0$,
there hold $\Phi_\vr (u)<0$ for all $u\in E_e\setminus B_R$ and
$\max\Phi_\vr (E_e)\leq C$.
\end{itemize}
\end{Lem}

\begin{proof}
Recall that $\jdz{u}_p^p\leq C_p\norm{u}^p$ for all $u\in E$ by
Sobolev's embedding theorem. 1) follows easily because, for $u\in E^+$
and $\de>0$ small enough
\[
\begin{aligned}
\Phi_\vr (u)&=\frac{1}{2}\norm{u}^2-\frac{\omega}{2}\jdz{u}_2^2
 -\Gamma_\vr (u)-\Psi_\vr (u)\\
 &\geq\frac{1}{2}\norm{u}^2-\frac{\omega}{2}\jdz{u}_2^2
 -\vr ^2C_1\jdz{Q}_\infty^2\norm{u}^4
 -{\jdz{P}_\infty}\bkt{\de\jdz{u}_2^2+c_\de'\jdz{u}_p^p}\\
\end{aligned}
\]
with $C_1$, $C_p$ independent of $u$ and $p>2$ (see \eqref{estimates
of Gamma-eps0} and \eqref{g estimates}).

For checking 2), take $e\in E^+\setminus\{0\}$. In virtue of
(\ref{lpdec}) and \eqref{g estimates}, one gets, for $u=se+v\in
E_e$,
\begin{equation}\label{linking ineq}
\begin{split}
\Phi_\vr (u)=&\,\frac{1}{2}\norm{se}^2-\frac{1}{2}\norm{v}^2
-\frac{\omega}{2}\jdz{u}_2^2
-\Gamma_\vr (u)-\Psi_\vr (u)\\
\leq&\, \frac{1}{2}s^2\norm{e}^2-\frac{1}{2}\norm{v}^2 - c_\de
d_\theta \inf P\cdot s^\theta\jdz{e}_\theta^\theta
\end{split}
\end{equation}
proving the conclusion.
\end{proof}

Recall that a sequence $\{u_n\}\subset E$ is called to be a
$(PS)_c$-sequence for functional $\Phi\in C^1(E,\R)$
if $\Phi(u_n)\to c$ and $\Phi'(u_n)\to 0$, and is called to be
$(C)_c$-sequence for $\Phi$ if $\Phi(u_n)\to c$ and
$(1+\|u_n\|)\Phi'(u_n)\to 0$. It is clear
that if $\{u_n\}$ is a $(PS)_c$-sequence with $\{\|u_n\|\}$ bounded
then it is also a $(C)_c$-sequence. Below we are going to study
$(C)_c$-sequences for $\Phi_\vr$ but firstly we observe the
following

\begin{Lem}\label{A(u) bdd}
Let $\{u_n\}\subset E\setminus\{0\}$ be bounded in
$L^\sigma(\R^3)$, where $\sigma>0$ is the constant in $(g_2)$.
Then $\hbt{\frac{A_{\vr ,u_n}^k}{\norm{u_n}}}$ is
bounded in $L^6(\R^3)$ uniformly in $\vr \in(0,1]$, for
$k=0,1,2,3$.
\end{Lem}

\begin{proof}
Set $v_n=\frac{u_n}{\norm{u_n}}$. Notice that
$A_{\vr ,u_n}^k$ satisfies the equation
$$
-\Delta A_{\vr ,u_n}^k=\vr ^24\pi Q_\vr (x)(\alpha_ku_n)\overline{u_n},
$$
hence,
$$
-\Delta \frac{A_{\vr ,u_n}^k}{\norm{u_n}}=\vr ^24\pi Q_\vr (x)
(\alpha_ku_n)\overline{v_n}.
$$
Observe that $\norm{v_n}=1$, $E$ embeds continuously into $L^q$ for
$q\in[2,3]$, and
\[
\begin{aligned}
\jdz{\int Q_\vr (x)(\alpha_ku_n)\overline{v_n}\cdot\psi dx}&\leq\jdz{Q}_\infty\jdz{u_n}_\sigma\jdz{v_n}_q\jdz{\psi}_6\\
 &\leq S^{-1/2}\jdz{Q}_\infty\jdz{u_n}_\sigma\jdz{v_n}_q
 \norm{\psi}_{\mathcal{D}}
\end{aligned}
\]
for any $\psi\in\mathcal{D}^{1,2}(\mathbb{R}^3,\mathbb{C}^4)$ and
$\frac{1}{\sigma}+\frac{1}{q}+\frac{1}{6}=1$. We infer
$$\norm{\frac{A_{\vr ,u_n}^k}{\norm{u_n}}}_{\mathcal{D}}
\leq \vr ^2\tilde{C}\jdz{Q}_\infty\jdz{u_n}_\sigma,$$ which yields
the conclusion.
\end{proof}

We now turn to an estimate on boundness of $(C)_c$-sequences which
is the key ingredient in the sequel. Recall that, by $(g_1)$, there
exist $r_1>0$ and $a_1>0$ such that
\begin{equation}\label{R8}
g(s)\leq \frac{a-|\omega|}{2\,|P|_\infty} \quad \text{for all $s\leq
r_1$},
\end{equation}
and, for $s\geq r_1$,  $g(s)\leq a_1 s^{p-2}$, so
$g(s)^{\sigma_0-1}\leq a_2 s^2$ with
$$
\sigma_0:=\frac{p}{p-2}>3
$$
which, jointly with $(g_2)$, yields (see \eqref{R9})
\begin{equation}\label{g-sigma0 estimate}
g(s)^{\sigma_0}\leq a_2 g(s)s^2\leq a_3 \widehat{G}(s)\quad
\text{for all $s\geq r_1$}.
\end{equation}

\begin{Lem}\label{PScseq eps estimate}
For any $\lambda>0$, denoting $I=[0,\lambda]$, there
is $\Lambda>0$ independent of $\vr $ such that, for all
$\vr \in(0,1]$, any $(C)_c$-sequence $\{u_n^\vr \}$ of
$\Phi_\vr $ with $c\in I$, there holds (up to a subsequence)
$$
\norm{u_n^\vr }\leq\Lambda
$$
for all $n\in\mathbb{N}$.
\end{Lem}

\begin{proof}
Let $\{u_n^\vr \}$ be a $(C)_c$-sequence of $\Phi_\vr $ with $c\in
I$: $\Phi_\vr(u^\vr_n)\to c$ and
$(1+\|u^\vr_n\|)\|\Phi'_\vr(u^\vr_n)\|\to 0$. Without loss of
generality we may assume that $\norm{u_n^\vr }\geq1$. The form of
$\Phi_\vr$ and the representation \eqref{R2} \,
($\Gamma'_\vr(u)u=4\Gamma_\vr(u)$) \, implies that
\begin{equation}\label{R3}
\begin{aligned}
2\lambda&>c+o(1)=\Phi_\vr (u_n^\vr ) -\frac{1}{2}\Phi_\vr '(u_n^\vr
)u_n^\vr =\Gamma_\vr (u_n^\vr)
 +\int P_\vr (x)\widehat{G}(|u^\vr_n|)\\
\end{aligned}
\end{equation}
and
\begin{equation}\label{R4}
\aligned o(1)=&\,\Phi'_\vr(u^\va_n)(u^{\vr+}_n-u^{\vr-}_n)\\
=&\,\|u^{\vr}_n\|^2-\omega\big(|u^{\vr+}_n|^2_2-|u^{\vr-}_n|^2_2\big)-
\Gamma'_\vr(u^{\vr}_n)(u^{\vr+}_n-u^{\vr-}_n)\\
&\, -\Re\int P_\vr(x)
 g(\jdz{u_n^\vr })u_n^\vr \cdot\overline{u_n^{\vr +}-u_n^{\vr -}}.
\endaligned
\end{equation}
By Lemma \ref{Gamma-eps nonnegative}, \eqref{R9} and \eqref{R3},
$\{u_n^\vr \}$ is bounded in $L^\sigma$ uniformly in $\vr $ with the
upper bound, denoted by $C_1$, depending on $\lambda$, $\sigma$,
$\theta$ and $\inf P$. It follows from \eqref{R4} that
\[
\aligned &\, o(1)+\frac{a-\jdz{\omega}}{a}\|u^\vr_n\|^2 \\ \leq &\,
\Gamma'_\vr(u_n^\vr )(u_n^{\vr+}
 -u_n^{\vr-})
 +\Re\int P_\vr (x)
 g(\jdz{u_n^\vr })u_n^\vr \cdot\overline{u_n^{\vr +}-u_n^{\vr -}}.
\endaligned
\]
This, together with \eqref{R8} and \eqref{l2ineq}, shows
\begin{equation}\label{R5}
\aligned &\, o(1)+\frac{a-\jdz{\omega}}{2a}\|u^\vr_n\|^2 \\ \leq &\,
\Gamma'_\vr(u_n^\vr )(u_n^{\vr+}
 -u_n^{\vr-})
 +\Re\int_{|u^\vr_n|\geq r_1} P_\vr (x)
 g(\jdz{u_n^\vr })u_n^\vr \cdot\overline{u_n^{\vr +}-u_n^{\vr -}}.
\endaligned
\end{equation}

Recall that $(g_1)$ and $(g_2)$ imply $2<\sigma\leq p$. Setting
$t=\frac{p\sigma}{2\sigma-p}$, one sees
\[
2 < t < p, \quad \frac{1}{\sigma_0}+\frac{1}{\sigma}+\frac{1}{t}=1.
\]
By H\"older's inequality, the fact $\Gamma_\vr(u^\vr_n)\geq 0$,
\eqref{g-sigma0 estimate}, \eqref{R3},  the boundedness of
$\{|u^\vr_n|_\sigma\}$ uniformly in $\vr$, and the embedding of $E$
into $L^t$, we have
\begin{equation}\label{R6}
\aligned &\, \int_{|u^\vr_n|\geq r_1} P_\vr(x) \, g(\jdz{u_n^\vr
})\jdz{u_n^\vr }\jdz{u_n^{\vr +}-u_n^{\vr
-}}\\
\leq &\, |P|_\infty\Big(\int_{|u^\vr_n|\geq r_1}
g(|u^\vr_n|)^{\sigma_0}\Big)^{1/\sigma_0} \Big(\int
|u^\vr_n|^\sigma\Big)^{1/\sigma} \Big(|u^{\vr+}_n-
u^{\vr-}_n|^t\Big)^{1/t}\\
\leq &\, C_2\|u^{\vr}_n\|
\endaligned
\end{equation}
with $C_2$ independent of $\vr$.

Let $q=\frac{6\sigma}{5\sigma-6}$. Then $2<q<3$ and
$\frac{1}{\sigma}+\frac{1}{q}+\frac{1}{6}=1$. Set
\[
\zeta=\left\{\aligned & 0 \, & \text{if $q=\sigma$};\\
& \frac{2(\sigma-q)}{q(\sigma-2)}  \, & \text{if $q<\sigma$};\\
& \frac{3(q-\sigma)}{q(3-\sigma)}  \, & \text{if $q>\sigma$}
\endaligned\right.
\]
and note that
$$
|u|_q \leq \left\{\aligned & \jdz{u}_2^{\zeta}
\cdot\jdz{u}_\sigma^{1-\zeta} \, & \mathrm{if}\ 2<q\leq \sigma\\
& \jdz{u}_3^{\zeta}\cdot\jdz{u}_\sigma^{1-\zeta} \, & \mathrm{if}\
\sigma< q<3.
\endaligned\right.
$$
By virtue of the H\"older inequality, Lemma \ref{Gamma-eps
nonnegative}, the boundedness of $\{|u^\vr_n|_\sigma\}$, and the
embedding of $E$ to $L^2$ and $L^3$, we obtain that
\[
\begin{aligned}
&\jdz{\Re\int Q_\vr (x)A_{\vr ,u_n^\vr }^k(\alpha_ku_n^\vr)
\cdot\overline{u_n^{\vr +}-u_n^{\vr -}}}   \\
=&\jdz{\|u^\vr_n\|\Re\int Q_\vr (x)\frac{A_{\vr ,u_n^\vr }^k}{\norm{u_n^\vr }}
(\alpha_ku_n^\vr) \cdot\overline{u_n^{\vr +}-u_n^{\vr -}}}  \\
\leq&\jdz{Q}_\infty \|u^\vr_n\| \jdz{\frac{A_{\vr ,u_n^\vr
}^k}{\norm{u_n^\vr }}}_6\jdz{u_n^\vr }_\sigma\jdz{u_n^{\vr
+}-u_n^{\vr -}}_q\\
\leq &\, \vr^2 C_3\|u^\vr_n\|\,|u^\vr_n|_q\,\leq \,\vr^2
C_4\|u^\vr_n\|^{1+\zeta}
\end{aligned}
\]
with $C_4$ independent of $\vr$. This, together with the
representation of \eqref{R2}, implies that
\begin{equation}\label{R7}
|\Gamma'_\vr(u^\vr_n)(u_n^{\vr +}-u_n^{\vr -})|\leq
C_5\|u^\vr_n\|^{1+\zeta}
\end{equation}
with $C_5$ independent of $\vr$.

Now the combination of \eqref{R5}, \eqref{R6} and \eqref{R7} shows
that
\begin{equation}\label{x1}
\|u^\vr_n\|^2\leq M_1\|u^\vr_n\|+M_2{\norm{u_n^\vr }^{1+\zeta}}
\end{equation}
with $M_1$ and $M_2$ being independent of $\vr\leq 1$. Therefore,
either $\|u^\vr_n\|\leq 1$ or there is $\Lambda\geq 1$ independent
of $\vr $ such that
\[
\|u^\vr_n\|\leq \Lambda
\]
as desired.
\end{proof}

Finally, for the later aim we define the operator $\mathcal{A}_{\vr
,k}:E\to\mathcal{D}^{1,2}(\mathbb{R}^3)$ by $\mathcal{A}_{\vr
,k}(u)=A_{\vr ,u}^k$. We have
\begin{Lem}\label{lemma of Aj}
For $k=0,1,2,3$,
\begin{itemize}
\item[$(1)$] $\mathcal{A}_{\vr ,k}$ maps bounded sets into bounded sets;
\item[$(2)$] $\mathcal{A}_{\vr ,k}$ is continuous;
\end{itemize}
\end{Lem}

\begin{proof}
Clearly, (1) is a straight consequence of (\ref{Ajsinequ}). (2)
follows easily because, for $u,v\in E$, one sees that $A_{\vr
,u}^j-A_{\vr ,v}^j$ satisfies
$$
-\Delta(A_{\vr ,u}^j-A_{\vr ,v}^j)=
\vr ^24\pi Q_\vr (x)\big[ (\alpha_ju)\bar{u}-(\alpha_jv)\bar{v} \big].
$$
Hence
\[
\begin{aligned}
\norm{A_{\vr ,u}^j-A_{\vr ,v}^j}_{\mathcal{D}^{1,2}}&\leq \vr ^2C
\jdz{Q}_\infty\big| (\alpha_ju)\bar{u}-(\alpha_jv)\bar{v} \big|_{6/5}\\
 &\leq \vr ^2C\jdz{Q}_\infty\Big(\jdz{u-v}_{12/5}\jdz{u}_{12/5}
 +\jdz{u-v}_{12/5}\jdz{v}_{12/5}\Big)\\
 &\leq \vr ^2C_1\jdz{Q}_\infty\bkt{\norm{u-v}\cdot\norm{u}+\norm{u-v}\cdot\norm{v}},
\end{aligned}
\]
and this implies the desired conclusion.
\end{proof}

\section{Preliminary results}\label{Preliminary results}

Observe that the non-local term $\Gamma_\vr$ is rather complex. The
main purpose of this section is, by cut-off arguments, to introduce
an auxiliary functional which will simplify our proofs.

\subsection{The limit equation}\label{subsec1}

In order to prove our main result, we
will make use of the limit equation. For any $\mu>0$, consider the
equation
$$
i\alpha\cdot\nabla u-a\beta u-\omega u=\mu g(\jdz{u})u.
$$
Its solutions are critical points of the functional
\[
\begin{aligned}
\mathscr{T}_\mu(u)&:=\frac{1}{2}\bkt{\norm{u^+}^2-\norm{u^-}^2
 -\omega\jdz{u}_2^2}-{\mu}\int G(\jdz{u})\\
 &=\frac{1}{2}\bkt{\norm{u^+}^2-\norm{u^-}^2
 -\omega\jdz{u}_2^2}-\mathscr{G}_\mu(u).
\end{aligned}
\]
defined for $u=u^++u^-\in E=E^+\oplus E^-$. Denote the critical
set, the least energy and the set of least energy solutions of
$\mathscr{T}_\mu$ as follows
\[
\aligned
&\mathscr{K}_\mu:=\{u\in E:\ \mathscr{T}_\mu'(u)=0\},\\
&\gamma_\mu:=\inf\{\mathscr{T}_\mu(u):\
 u\in\mathscr{K}_\mu\setminus\{0\}\},\\
&\mathscr{R}_\mu:=\{u\in\mathscr{K}_\mu:\
 \mathscr{T}_\mu(u)=\gamma_\mu,\
 \jdz{u(0)}=\jdz{u}_\infty\}.
\endaligned
\]
The following lemma is from \cite{Ding2} (see also \cite{Ding2008})

\begin{Lem}\label{ding2008}
There hold the following:\\
$i)$ $\mathscr{K}_\mu\not=\emptyset$, $\gamma_\mu>0$
 and $\mathscr{K}_\mu\subset\cap_{q\geq2}W^{1,q}$,\\
$ii)$ $\gamma_\mu$ is attained and $\mathscr{R}_\mu$
 is compact in $H^1(\mathbb{R}^3,\mathbb{C}^4)$,\\
$iii)$ there exist $C,c>0$ such that
$$\jdz{u(x)}\leq C\exp\bkt{-c\jdz{x}}$$
for all $x\in\mathbb{R}^3$ and $u\in\mathscr{R}_\mu$.
\end{Lem}

Motivated by Ackermann \cite{Ackermann} (also see
\cite{Ding2010,Ding2012,Ding2008}), for a fixed $u\in E^+$,
let $\varphi_u:E^-\to\mathbb{R}$
defined by $\varphi_u(v)=\mathscr{T}_\mu(u+v)$. We have, for any
$v,w\in E^-$,
$$
\varphi_u''(v)[w,w]=-\norm{w}^2-\omega\jdz{w}_2^2
 -\mathscr{G}_\mu''(u+v)[w,w]\leq-\norm{w}^2.
$$
In addition
$$
\varphi_u(v)\leq\frac{a+\jdz{\omega}}{2a}\norm{u}^2
 -\frac{a-\jdz{\omega}}{2a}\norm{v}^2.
$$
Therefore, there exists a unique $\mathscr{J}_\mu:E^+\to E^-$ such that
$$
\mathscr{T}_\mu(u+\mathscr{J}_\mu(u))=\max_{v\in E^-}\mathscr{T}_\mu(u+v).
$$
Define
$$
J_\mu: E^+\to\mathbb{R},\ \ J_\mu(u)=\mathscr{T}_\mu(u+\mathscr{J}_\mu(u)),
$$
$$
\mathscr{M}_\mu:=\{u\in E^+\setminus\{0\}:\
 J_\mu'(u)u=0\}.
$$
Plainly, critical points of $J_\mu$ and $\mathscr{T}_\mu$ are in one-to-one
correspondence via the injective map $u\mapsto u+\mathscr{J}_\mu(u)$ from
$E^+$ into $E$. For any $u\in E^+$ and $v\in E^-$, setting $z=v-\mathscr{J}_\mu(u)$ and
$l(t)=\mathscr{T}_\mu(u+\mathscr{J}_\mu(u)+tz)$, one has
$l(1)=\mathscr{T}_\mu(u+v)$, $l(0)=\mathscr{T}_\mu(u+\mathscr{J}_\mu(u))$ and
$l'(0)=0$. Thus $l(1)-l(0)=\int_0^1(1-t)l''(t)dt$. This implies that
\[
\begin{aligned}
&\mathscr{T}_\mu(u+v)-\mathscr{T}_\mu(u+\mathscr{J}_\mu(u))\\
=&\int_0^1(1-t)\mathscr{T}_\mu''\bkt{u+\mathscr{J}_\mu(u)-tz}[z,z]dt\\
=&-\int_0^1(1-t)\bkt{\norm{z}^2+\omega\jdz{z}_2^2}dt
-\int_0^1(1-t)\mathscr{G}_\mu''(u+\mathscr{J}_\mu(u)-tz)[z,z]dt,
\end{aligned}
\]
hence
\begin{equation}\label{the equa 1}
\begin{split}
 &\int_0^1(1-t)\mathscr{G}_\mu''(u+\mathscr{J}_\mu(u)-tz)[z,z]dt \\
 &+\frac{1}{2}\norm{z}^2+\frac{\omega}{2}\jdz{z}_2^2
 =\mathscr{T}_\mu(u+\mathscr{J}_\mu(u))-\mathscr{T}_\mu(u+v).
\end{split}
\end{equation}
It is not difficult to see that, for each $u\in E^+\setminus\{0\}$
there is a unique $t=t(u)>0$ such that $tu\in\mathscr{M}_\mu$ and
$$
\gamma_\mu=\inf\{J_\mu(u):\ u\in\mathscr{M}_\mu\}= \inf_{e\in
E^+\setminus\{0\}}\max_{u\in E_e}\mathscr{T}_\mu(u)
$$
(see \cite{Ding2008}, \cite{Ding2010}). The following lemma is from \cite{Ding2010}.
\begin{Lem}\label{ding2010}
There hold:
\begin{itemize}
\item[$1)$.] Let $u\in\mathscr{M}_\mu$ be such that $J_\mu(u)=\gamma_\mu$
and set $E_u=E^-\oplus\R^+u$. Then
$$
\max_{w\in E_u}\mathscr{T}_\mu(w)=J_\mu(u).
$$
\item[$2)$.] If $\mu_1<\mu_2$, then $\gamma_{\mu_1}>\gamma_{\mu_2}$.
\end{itemize}
\end{Lem}

\subsection{Auxiliary functionals}

In order to make the reduction method work for $\Phi_\vr$
as $\vr$ small, we circumvent by cutting off the nonlocal
terms. We find our current framework is more delicate, since
the solutions we look for are at the least energy level and
$\Gamma_\vr$ is not convex (even for $u$ with $\|u\|$ large).
By cutting off  the nonlocal terms, and using the reduction
method, we are able to find a critical point via an
appropriate min-max scheme. The critical point will eventually
be shown to be a least energy solution of the original
equation when $\vr$ is sufficiently small.

By virtue of $(P_0)$, set $\mu=b:=\inf P(x)>0$, take $e_0\in\mathscr{M}_b$ such that
$J_b(e_0)=\gamma_b$, and set $E_{e_0}=E^-\oplus\R^+ e_0$. One
has
\begin{Lem}\label{max Phi-eps < gamma-b}
For all $\vr >0$, \, $\max\limits_{v\in E_{e_0}}\Phi_\vr
(v)\leq\gamma_b$.
\end{Lem}
\begin{proof}
It is clear that $\Phi_\vr (u)\leq \mathscr{T}_b(u)$ for all
$u\in E$, hence, by Lemma \ref{ding2010}
$$
\max_{v\in E_{e_0}}\Phi_\vr (v)\leq\max_{v\in
E_{e_0}}\mathscr{T}_b(v)=J_b(e_0)=\gamma_b
$$
as claimed.
\end{proof}

To introduce the modified functional, by virtue of Lemma \ref{PScseq
eps estimate}, for $\lm=\ga_{b}$  and  $I=[0, \ga_{b}]$, let
$\Lambda\geq 1$ be the associated constant (independent of $\vr$).
Denote $T:=(\Lambda+1)^2$ and let $\eta:[0,\infty)\to[0,1]$ be a
smooth function with $\eta(t)=1$ if $0\leq t\leq T$, $\eta(t)=0$ if
$t\geq T+1$, $\max\jdz{\eta'(t)}\leq c_1$ and
$\max\jdz{\eta''(t)}\leq c_2$. Define
\[
\begin{aligned}
\widetilde{\Phi}_\vr
(u)&=\frac{1}{2}\bkt{\norm{u^+}^2-\norm{u^-}^2-\omega\jdz{u}_2^2}-\eta(\norm{u}^2)\Gamma_\vr
(u)
-\Psi_\vr (u)\\
 &=\frac{1}{2}\bkt{\norm{u^+}^2-\norm{u^-}^2-\omega\jdz{u}_2^2}-\mathscr{F}_\vr (u)-\Psi_\vr (u).
\end{aligned}
\]
By definition, $\Phi_\vr|_{B_T}=\widetilde{\Phi}_\vr|_{B_T}$. It is
easy to see that $0\leq {\mathscr{F}_\vr (u)}\leq {\Gamma_\vr (u)}$
and
$$
\jdz{\mathscr{F}_\vr '(u)v}\leq\jdz{2\eta'(\norm{u}^2)\Gamma_\vr
(u)\inp{u}{v}} +\jdz{\Gamma_\vr '(u)v}
$$
for $u,v\in E$.

\medskip

\begin{Lem}\label{lemR1}
There exists $\vr_1>0$ such that, for any $\vr\leq \vr_1$, if
$\{u^\vr_n\}$ is a $(C)_c$ sequence of $\widetilde{\Phi}_\vr$ with
$c\in I$ then $\|u^\vr_n\|\leq \Lambda+\frac12$, and consequently
$\widetilde{\Phi}_\vr(u^\vr_n)=\Phi_\vr(u^\vr_n)$.
\end{Lem}

\begin{proof}
We repeat the arguments of Lemma \ref{PScseq
eps estimate}. Let $\{u^\vr_n\}$ be a $(C)_c$-sequence  of
$\widetilde{\Phi}_\vr $ with $c\in I$. If $\|u^\vr_n\|^2\geq T+1$
then $\widetilde{\Phi}_\vr(u^\vr_n)=\Phi_\vr(u^\vr_n)$ so, by Lemma
\ref{PScseq eps estimate}, one has $\|u^\vr_n\|\leq \Lambda$, a
contradiction. Thus we assume that $\|u^\vr_n\|^2\leq T+1$. Then,
using \eqref{estimates of Gamma-eps0},  $ |\eta'(\|u^\vr_n\|^2)
\|u^\vr_n\|^2 \Gamma_\vr(u^\vr_n)| \leq \vr^2 d_1 $ (here and in the
following, by $d_j$ we denote positive constants independent of
$\vr$). Similarly to \eqref{R3},
\[
2\ga_b>c+o(1)\geq
\big(\eta(\|u^\vr_n\|^2)+2\eta'(\|u^\vr_n\|^2)\|u^\vr_n\|^2\big)\Gamma_\vr(u^\vr_n)+
\int P_\vr(x)\widehat{G}(|u^\vr_n|)
\]
which yields
\[
2\ga_b+\vr^2d_1> \eta(\|u^\vr_n\|^2)\Gamma_\vr(u^\vr_n)+ \int
P_\vr(x)\widehat{G}(|u^\vr_n|),
\]
consequently $|u^\vr_n|_\sigma\leq d_2$. Similarly to \eqref{R5} we
get that
\[
\aligned
\frac{a-\jdz{\omega}}{2a}\|u^\vr_n\|^2 \leq &\,\vr^2d_3 +
 \eta(\|u^\vr_n\|^2)\Gamma'_\vr(u_n^\vr)(u_n^{\vr+}-u_n^{\vr-})\\
 &\,+\Re\int_{|u^\vr_n|\geq r_1} P_\vr (x)
 g(\jdz{u_n^\vr })u_n^\vr \cdot\overline{u_n^{\vr +}-u_n^{\vr -}}
\endaligned
\]
which, together with \eqref{R6} and \eqref{R7}, implies either
$\|u^\vr_n\|\leq 1$ or as \eqref{x1}
\[
\|u^\vr_n\|^2\leq \vr^2 d_4+M_1\|u^\vr_n\|+M_2\|u^\vr_n\|^{1+\zeta},
\]
thus
\[
\|u^\vr_n\|\leq \vr^2d_5+\Lambda.
\]
The proof is complete.
\end{proof}

Based on this lemma, to prove Theorem \ref{main theorem} it suffices
to study $\widetilde{\Phi}_\vr $ and get its critical points with
critical values in $[0,\ga_{b}]$. This will be done via a series of
arguments. The first is to introduce the minimax values of
$\widetilde{\Phi}_\vr $. It is easy to verify the following lemma.

\begin{Lem}\label{tilde Phi-eps linking}
$\widetilde{\Phi}_\vr $ possesses a linking structure and
we can replace $\Phi_\vr$ by $\widetilde{\Phi}_\vr$
in Lemma \ref{max Phi-eps<C}. In addition,
$$
\max_{v\in E_{e_0}}\widetilde{\Phi}_\vr (v)\leq\ga_b,
$$
where $e_0\in\mathscr{M}_b$ such that $J_b(e_0)=\ga_b$ and
$E_{e_0}=E^-\oplus\R^+ e_0$
\end{Lem}

\begin{proof}
One can follow the proofs of Lemmas \ref{max Phi-eps<C}
and Lemma \ref{max Phi-eps < gamma-b} with minor changes.
\end{proof}

Define (see \cite{Ding1,Szulkin})
$$
c_\vr :=\inf_{e\in E^+\setminus\{0\}}\max_{u\in E_e}\widetilde{\Phi}_\vr (u).
$$
As a consequence of Lemma \ref{tilde Phi-eps linking} we have
\begin{Lem}\label{C independent of eps}
$\tau\leq c_\vr \leq\gamma_b$.
\end{Lem}


We now describe further the minimax value $c_\vr$. As before, for a
fixed $u\in E^+$ we define $\phi_u:E^-\to\mathbb{R}$ by
$$\phi_u(v)=\widetilde{\Phi}_\vr (u+v).$$
A direct computation gives, for any $v,z\in E^-$,
\[
\begin{aligned}
\phi_u''(v)[z,z]&=-\norm{z}^2-\omega\jdz{z}_2^2
-\mathscr{F}_\vr'' (u+v)[z,z]-\Psi_\vr'' (u+v)[z,z],\\
 &\leq-\frac{(a-\jdz{\omega})}{a}\norm{z}^2-\mathscr{F}_\vr'' (u+v)[z,z],
\end{aligned}
\]
and
\[
\begin{aligned}
\mathscr{F}&_\vr'' (u+v)[z,z]\\
=&\bkt{4\eta''(\norm{u+v}^2)\jdz{\inp{u+v}{z}}^2
+2\eta'(\norm{u+v}^2)\norm{z}^2}\Gamma_\vr (u+v)\\
&+4\eta'(\norm{u+v}^2)\inp{u+v}{z}\Gamma_\vr '(u+v)z\\
&+\eta(\norm{u+v}^2)\Gamma_\vr ''(u+v)[z,z].
\end{aligned}
\]
Combining \eqref{estimates of Gamma-eps0}-\eqref{estimates of
Gamma-eps2} yields that there exists $\vr _0\in(0,\vr _1]$ such that
$$
\phi_u''(v)[z,z]\leq-\frac{a-\jdz{\omega}}{2a}\norm{z}^2 \ \ \ \
\mathrm{if}\ 0<\vr \leq\vr _0.
$$
Since
$$
\phi_u(v)\leq\frac{a+\jdz{\omega}}{2a}\norm{u}^2-\frac{a-\jdz{\omega}}{2a}\norm{v}^2,
$$
there is $h_\vr :E^+\to E^-$, uniquely defined, such that
$$
\phi_u(h_\vr (u))=\max_{v\in E^-}\phi_u(v)
$$
and
$$
v\not=h_\vr (u)\Leftrightarrow\widetilde{\Phi}_\vr (u+v)
<\widetilde{\Phi}_\vr (u+h_\vr (u)).
$$
It is clear that, for all $v\in E^-$,
$
0=\phi_u'(h_\vr (u))v
$.
Observe that, similarly to (\ref{the equa 1}), we have for $u\in E^+$
and $v\in E^-$
\begin{equation}\label{the equa 2}
\begin{split}
&\widetilde{\Phi}_\vr (u+h_\vr (u))-\widetilde{\Phi}_\vr (u+v)\\
 =&\int_0^1(1-t)\Big[\mathscr{F}_\vr ''(u+h_\vr (u)
 +t(v-h_\vr (u)))[v-h_\vr (u),v-h_\vr (u)]\\
 &+\Psi_\vr ''(u+h_\vr (u)+t(v-h_\vr (u)))
 [v-h_\vr (u),v-h_\vr (u)]\Big]dt\\
 &+\frac{1}{2}\norm{v-h_\vr (u)}^2+\frac{\omega}{2}\jdz{v-h_\vr (u)}_2^2.
\end{split}
\end{equation}

Define $I_\vr :E^+\to\mathbb{R}$ by
$$
I_\vr (u)=\widetilde{\Phi}_\vr (u+h_\vr (u)),
$$
and set
$$
\mathscr{N}_\vr :=\{u\in E^+\setminus\{0\}:\ I_\vr '(u)u=0\}.
$$

\begin{Lem}\label{unique t}
For any $u\in E^+\setminus\{0\}$, there is a unique $t=t(u)>0$ such
that $tu\in\mathscr{N}_\vr $.
\end{Lem}

\begin{proof}
This proof is quite technical, for details
we refer \cite{Ackermann,Ding2008}.
We only give a sketch of the proof. Firstly,
we observe that for any $u\in E\setminus\{0\}$ and $v\in E$,
\[
\aligned
&\big( \Ga_\vr''(u)[u,u]-\Ga_\vr'(u)u \big)
+2\big( \Ga_\vr''(u)[u,v]-\Ga_\vr'(u)v \big) +\Ga_\vr''(u)[v,v]  \\
=&\, 2\,\vr^2 \iint\frac{Q_\vr(x)Q_\vr(y)}{|x-y|}
\Big[ J_0(x)[\al_0(u+v)\ov{(u+v)}](y)   \\
&\, -\sum_{k=1}^3 J_k(x)[\al_k(u+v)\ov{(u+v)}](y) \Big] dxdy   \\
&\, +2\,\vr^2\iint\frac{Q_\vr(x)Q_\vr(y)}{|x-y|}\Big[\big(\Re[\al_0u\ov
v(x)]\big)\big(\Re[\al_0u\ov v(y)]\big)   \\
&\,\qquad -\sum^3_{k=1}\big(\Re[\al_ku\ov
v(x)]\big)\big(\Re[\al_ku\ov v(y)]\big) \Big]dxdy  \\
\geq&\, O(\vr^2)\|u\|^2\|v\|^2.
\endaligned
\]
Here we used the formula
\[
\aligned
&\pm(\bt z, z)(y)(\bt u, u)(x) +
\sum_{k=1}^3(\bt z, \pi_k z)(y)(\bt u, \pi_k u)(x)   \\
=&\,\pm(\bt z, z)(y)(\bt u, u)(x) +
\sum_{k=1}^3(\al_k z, z)(y)(\al_k u, u)(x)   \\
\leq&\, |z(y)|_{\C^4}^2|u(x)|_{\C^4}^2
\endaligned
\]
which follows from \eqref{XX1} and \eqref{XX2} with $z=u+v\in E$.
Consequently, we deduce that
\[
\aligned
&\big( \mathscr{F}_\vr''(u)[u,u]-\mathscr{F}_\vr'(u)u \big)
+2\big( \mathscr{F}_\vr''(u)[u,v]-\mathscr{F}_\vr'(u)v \big)  \\
&+\mathscr{F}_\vr''(u)[v,v]
\geq  o(\vr)\|v\|^2 + o(\vr) \, .
\endaligned
\]
Invoking the arguments in \cite{Ackermann},
if $z\in E^+\setminus\{0\}$ with
$I_\vr'(z)z=0$, we see by a delicate calculation that,
for $\vr$ sufficiently small,
\begin{equation}\label{XX4}
I_\vr''(z)[z,z]<0.
\end{equation}

Now for a fixed $u\in E^+\setminus\{0\}$,
we set $f(t)=I_\vr(tu)$. From Lemma \ref{tilde Phi-eps linking},
we see that $f(0)=0$, $f(t)>0$ for $t>0$ sufficiently small,
and $f(t)\to-\infty$ as $t\to\infty$.
Thus there exists $t(u)>0$ such that
\[
I_\vr(t(u)u)=\sup_{t\geq0}I_\vr(tu).
\]
It is clear that
\[
\frac{d\, I_\vr(tu)}{dt}\bigg|_{t=t(u)}=
I_\vr'(t(u)u)u=\frac{1}{t(u)}I_\vr'(t(u)u)t(u)u=0,
\]
and consequently by \eqref{XX4}
\[
I_\vr''(t(u)u)[t(u)u, t(u)u]<0.
\]
Therefore, one sees that such $t(u)>0$ is unique.
\end{proof}

\begin{Lem}\label{d-eps=c-eps}
$c_\vr =\inf_{u\in\mathscr{N}_\vr }I_\vr (u)$.
\end{Lem}

\begin{proof}
Indeed, denoting $\hat c_\vr =\inf_{u\in\mathscr{N}_\vr }I_\vr (u)$,
given $e\in E^+$, if $u=v+se\in E_e$ with $\widetilde{\Phi}_\vr
(u)=\max_{z\in E_e}\widetilde{\Phi}_\vr (z)$ then the restriction
$\widetilde{\Phi}_\vr |_{E_e}$ of $\widetilde{\Phi}_\vr $ on $E_e$
satisfies $(\widetilde{\Phi}_\vr |_{E_e})'(u)=0$ which implies
$v=h_\vr (se)$ and $I_\vr '(se)(se)=0$, i.e. $se\in\mathscr{N}_\vr
$. Thus $\hat c_\vr \leq c_\vr $. On the other hand, if
$w\in\mathscr{N}_\vr $ then $(\widetilde{\Phi}_\vr |_{E_w})'(w+h_\vr
(w))=0$, hence, $c_\vr \leq\max_{u\in E_w}\widetilde{\Phi}_\vr
(u)=I_\vr (w)$. Thus $\hat c_\vr \geq c_\vr $.
\end{proof}

\begin{Lem}\label{For any e  Te} For any $e\in E^+\setminus\{0\}$,
there is $T_e>0$ independent of $\vr $ such that
$t_\vr \leq T_e$ for $t_\vr >0$  satisfying
$t_\vr  e\in\mathscr{N}_\vr $.
\end{Lem}

\begin{proof}
Since $I_\vr '(t_\vr  e)(t_\vr  e)=0$, one
gets
$$
\widetilde{\Phi}_\vr (t_\vr  e+h_\vr (t_\vr  e))
=\max_{w\in E_e}\widetilde{\Phi}_\vr (w)\geq\tau.
$$
This,
together with Lemma \ref{tilde Phi-eps linking}, shows the
assertion.
\end{proof}

Let $\mathscr{K}_\vr :=\{u\in
E:\ \widetilde{\Phi}_\vr '(u)=0\}$ be the critical set of
$\widetilde{\Phi}_\vr $. Since critical points of $I_\vr$ and
$\widetilde{\Phi}_\vr$ are in one-to-one
correspondence via the injective map $u\mapsto u+h_\vr(u)$ from
$E^+$ into $E$, let us denoted by
$$
\mathscr{C}_\vr :=\{u\in
\mathscr{K}_\vr:\ \widetilde{\Phi}_\vr(u)=c_\vr\},
$$
from Lemma \ref{d-eps=c-eps},
one easily sees that if $\mathscr{C}_\vr\neq\emptyset$ then
$c_\vr=\inf\big\{
\widetilde{\Phi}_\vr (u):\ u\in\mathscr{K}_\vr
\setminus\{0\}
\big\}$.
Next we estimate the regularity of critical points of
$\widetilde{\Phi}_\vr$. By using the
same iterative argument of \cite{Sere2} one obtains easily the
following

\begin{Lem}\label{critical points in W1,q}
If $u\in\mathscr{K}_\vr$ with $|\widetilde{\Phi}_\vr (u)|\leq C_1$,
then, for any $q\in[2,+\infty)$, $u\in
W^{1,q}(\mathbb{R}^3,\mathbb{C}^4)$ with
$\norm{u}_{W^{1,q}}\leq\Lambda_q$ where $\Lambda_q$ depends only on
$C_1$ and $q$.
\end{Lem}

\begin{proof}
See \cite{Sere2}. We outline the proof as follows. From (\ref{D22}),
we write
\[
\begin{aligned}
u&=H_0^{-1}\Big( \om u+Q_\vr (x)A_{\vr ,u}^0u -\sum_{k=1}^3Q_\vr
(x)\al_kA_{\vr ,u}^ku+P_\vr (x)g(\jdz{u})u  \Big).
\end{aligned}
\]
For the later use, let $\rho:[0,\infty)\to[0,1]$
be a smooth function satisfying $\rho(s)=1$ if $s\in[0,1]$ and
$\rho(s)=0$ if $s\in[2,\infty)$. Then we have
\[
\begin{aligned}
g(s):=&\,g_1(s)+g_2(s)\\
=&\,\rho(s)g(s)+(1-\rho(s))g(s).
\end{aligned}
\]
Consequently,
$u=u_1+u_2+u_3$ with
\[
\begin{aligned}
u_1=&H_0^{-1}\bkt{\omega u+P_\vr \cdot g_1(\jdz{u})u},\\
u_2=&H_0^{-1}\bkt{Q_\vr \cdot A_{\vr ,u}^0u-\sum Q_\vr \cdot \alpha_kA_{\vr ,u}^ku},\\
u_3=&H_0^{-1}\bkt{P_\vr \cdot g_2(\jdz{u})u}.
\end{aligned}
\]
Noting that, by H\"{o}lder's inequality, for $q\geq2$
$$
\big|Q_\vr \alpha_kA_{\vr ,u}^ku\big|_s\leq\jdz{Q}_\infty\big|A_{\vr
,u}^k\big|_6\jdz{u}_q
$$
with $\frac{1}{s}=\frac{1}{6}+\frac{1}{q}$ and since
$$
\big|P_\vr
(x)\jdz{u}^{p-2}u\big|_t\leq\jdz{P}_\infty\jdz{u}_{t(p-1)}^{p-1},
$$
one has
$$
u_1\in W^{1,2}\cap W^{1,3},\ u_2\in W^{1,s},\ u_3\in W^{1,t}.
$$
Then, denoting $s^*=\frac{3s}{3-s}$ and $t^*=\frac{3t}{3-t}$,
$u\in W^{1,q}$ with $q=\min\{s^*,t^*\}$. A standard bootstrap argument shows
that $u\in\cap_{q\geq2}L^q$,
$u_1\in\cap_{q\geq2}W^{1,q}$, $u_2\in\cap_{6>q\geq2}W^{1,q}$ and
$u_3\in\cap_{q\geq2}W^{1,q}$.

By the Sobolev's embedding theorems, $u\in C^{0,\gamma}$ for some
$\gamma\in(0,1)$. This, together with elliptic regularity (see
\cite{Trudinger}), shows $A_{\vr ,u}^k\in W^{2,6}_{loc}\cap L^6$ for
$k=0,1,2,3$ and
$$
\big\| A_{\vr ,u}^k \big\|_{W^{2,6}(B_1(x))}\leq C_2 \Big(\vr
^2\jdz{Q}_\infty\jdz{u}_{L^{12}(B_2(x))}^2 +\big| A_{\vr
,u}^k \big|_{L^6(B_2(x))}\Big)
$$
for all $x\in\mathbb{R}^3$, with $C_2$ independent of $x$ and
$\vr $, where $B_r(x)=\{y\in\mathbb{R}^3:\jdz{y-x}<r\}$ for
$r>0$. Since $W^{2,6}(B_1(x))\hookrightarrow C^1({B_1(x)})$, we have
\begin{equation}\label{Ak local estimates}
\big\| A_{\vr ,u}^k \big\|_{C^{1}(B_1(x))}\leq C_3\Big(\vr ^2\jdz{Q}_\infty
\jdz{u}_{L^{12}(B_2(x))}^2+\big| A_{\vr ,u}^k \big|_{L^6(B_2(x))}\Big)
\end{equation}
for all $x\in\mathbb{R}^3$ with $C_3$ independent of $x$ and
$\vr $. Consequently $A_{\vr ,u}^k\in L^\infty$, and that yields
$$
\big| Q_\vr \alpha_kA_{\vr ,u}^ku \big|_s\leq\jdz{Q}_\infty
\big| A_{\vr ,u}^k \big|_\infty\jdz{u}_s.
$$
Thus $u_2\in\cap_{q\geq2}W^{1,q}$, and combining with
$u_1,u_3\in\cap_{q\geq2}W^{1,q}$ the conclusion is obtained.
\end{proof}

\begin{Rem}
Let $\mathscr{L}_\vr$ denote the set of all least energy solutions
of $\widetilde{\Phi}_\vr $. If $u\in\mathscr{L}_\vr$,
$\widetilde{\Phi}_\vr (u)=c_\vr \leq\gamma_b$. Recall that
$\mathscr{L}_\vr$ is bounded in $E$ with an upper bound $\Lambda$
independent of $\vr $. Therefore, as a consequence of Lemma
\ref{critical points in W1,q} we see that, for each
$q\in[2,+\infty)$ there is $C_q>$ independent of $\vr $ such that
\begin{equation}\label{Cq estimate of L}
\norm{u}_{W^{1,q}}\leq C_q\ \ \ \ \mathrm{for\ all}\
u\in\mathscr{L}_\vr.
\end{equation}
This, together with the Sobolev embedding theorem, implies that
there is $C_\infty>0$ independent of
$\vr $ with
\begin{equation}\label{C-infty estimate of L}
\norm{u}_{\infty}\leq C_\infty\ \ \ \ \mathrm{for\ all}\
u\in\mathscr{L}_\vr.
\end{equation}
\end{Rem}

\section{Proof of the main result}

Throughout this section we suppose $\omega\in(-a,a)$ and
that $(g_1)$-$(g_2)$, $(Q_0)$, $(P_0)$ are satisfied,
and recall that we always assume
$0\in\mathscr{P}$. The main theorem will be carried in
three parts: {\it Existence, Concentration, and Exponential decay}.


\subsection*{Part 1. Existence}
Keeping the notation of Section \ref{Preliminary results} we now turn
to the existence result of the main theorem.
The proof is carried out in three lemmas. The modified problem gives
us an access to Lemma \ref{d to gamma}, which is the key ingredient
for Lemma \ref{d-eps is attained}.

Recall that $\ga_m$ denotes the least energy of $\mathscr{T}_m$ (see the
subsection \ref{subsec1}), where
$\mu=m:=\max_{x\in\R^3}P(x)$, and $J_m$ denotes the associated reduction
functional on $E^+$. We have

\begin{Lem}\label{d to gamma}
$c_\vr \to\ga_m$ as $\vr \to0$.
\end{Lem}

\begin{Lem}\label{d-eps is attained}
$c_\vr $ is attained for all small $\vr >0$.
\end{Lem}

\begin{Lem}\label{least energy solution compact}
$\mathscr{L}_\vr $ is compact in $W^{1,q}$ for each $q\geq 2$, for
all small $\vr >0$.
\end{Lem}

\begin{proof}[Proof of Lemma \ref{d to gamma}] Firstly we show that
\begin{equation}\label{d-eps>gamma}
\liminf\limits_{\vr \to0}c_\vr \geq\ga_m.
\end{equation}
Arguing indirectly, assume that $\liminf_{\vr \to0}\,c_\vr
<\ga_m$. By the definition of $c_\vr$ and Lemma \ref{d-eps=c-eps}
we can choose an $e_j\in \mathscr{N}_\vr$ and $\delta>0$ such that
$$
\max_{u\in E_{e_j}}\widetilde{\Phi}_{\vr _j}(u)\leq\ga_m-\delta
$$
as $\vr _j\to0$. Since $P_\vr (x)\leq m$ and $\mathscr{F} (u)=o(1)$
as $\vr\to 0$ uniformly in $u$ (by \eqref{estimates of Gamma-eps0}
and the definition of $\eta$), the representations of
$\widetilde{\Phi}_{\vr}$ and $\mathscr{T}_m$ imply that $\widetilde{\Phi}_\vr
(u)\geq \mathscr{T}_m(u)-\de/2$ for all $u\in E$ and $\vr$ small. Note also
that $\ga_m\leq J_m(e_j) \leq \max_{u\in E_{e_j}}\mathscr{T}_m(u)$.
Therefore we get, for all $\vr_j$ small,
$$
\ga_m-\delta\geq\max_{u\in E_{e_j}} \widetilde{\Phi}_{\vr _j}(u)
\geq\max_{u\in E_{e_j}}\mathscr{T}_m(u)-\frac{\de}{2}\geq\ga_m-\frac{\de}{2},
$$
a contradiction.

\medskip

We now turn to prove the desired conclusion. Set $P^0(x)=m-P(x)$ and
$P^0_\vr (x)=P^0(\vr  x)$. Then
\begin{equation}\label{tilde-Phi-eps----Gm}
\widetilde{\Phi}_\vr (u)=\mathscr{T}_m(u)-\mathscr{F}_\vr (u)+\int P^0_\vr
(x)G(\jdz{u}).
\end{equation}

In virtue of Lemma \ref{ding2008}, let $u=u^++u^-\in\mathscr{R}_m$,
be a least energy solution of the limit equation with $\mu=m$, and set
$e=u^+$. Clearly, $e\in\mathscr{M}_m$, $\mathscr{J}_m(e)=u^-$ and
$J_m(e)=\ga_m$. There is a unique $t_\vr >0$ such that $t_\vr
e\in\mathscr{N}_\vr $ and one has
\begin{equation}\label{d-eps<I-eps(t-eps)}
c_\vr \leq I_\vr (t_\vr  e).
\end{equation}
By Lemma \ref{For any e  Te} $t_\vr $ is bounded. Hence, without
loss of generality we can assume $t_\vr \to t_0$ as $\vr \to0$.
Using (\ref{the equa 1}) and (\ref{the equa 2}), we infer
\[
\begin{aligned}
&\frac{1}{2}\norm{v_\vr }^2+(I)
=\ \widetilde{\Phi}_\vr (w_\vr)-\widetilde{\Phi}_\vr (u_\vr)\\
=&\ \mathscr{T}_m(w_\vr)-\mathscr{T}_m(u_\vr)-\mathscr{F}_\vr (w_\vr)
+\mathscr{F}_\vr (u_\vr)\\
& +\int P^0_\vr (x)G(\jdz{w_\vr})-\int P^0_\vr (x)G(\jdz{u_\vr})
\end{aligned}
\]
where, setting
\[
u_\vr=t_\vr e+\mathscr{J}_m(t_\vr e), \  w_\vr=t_\vr e+h_\vr
(t_\vr e), \  v_\vr =u_\vr-w_\vr,
\]
\[
(I):=\frac{\omega}{2}\jdz{v_\vr }_2^2+ \int_0^1(1-s)
\big(\mathscr{F}_\vr ''(w_\vr+sv_\vr )[v_\vr ,v_\vr ]
+\Psi_\vr ''(w_\vr+sv_\vr )[v_\vr ,v_\vr
]\big)dt.
\]
Taking into account that
\[
\mathscr{F}_\vr (u_\vr)-\mathscr{F}_\vr (w_\vr)
=\mathscr{F}_\vr '(w_\vr)v_\vr
+\int_0^1(1-s)\mathscr{F}_\vr ''(w_\vr+sv_\vr
)[v_\vr ,v_\vr ]dt
\]
and
\[
\begin{aligned}
&\int P^0_\vr (x)\big(G(\jdz{w_\vr})-G(\jdz{u_\vr})\big)\\
=&-\int P^0_\vr (x)g(\jdz{u_\vr})u_\vr \cdot\overline{v_\vr }
+\int_0^1(1-s)\mathscr{G}_m''(u_\vr-sv_\vr )[v_\vr ,v_\vr ]dt\\
&-\int_0^1(1-s)\Psi_\vr ''(u_\vr-sv_\vr )[v_\vr
,v_\vr ]dt,
\end{aligned}
\]
setting
\[
\begin{aligned}
(II):=&\int_0^1(1-s)\Psi_\vr ''(u_\vr-sv_\vr
)[v_\vr ,v_\vr ]dt,
\end{aligned}
\]
one has
\[
\begin{aligned}
&\frac{1}{2}\norm{v_\vr }^2+(I)+(II)\\
\leq&\,\mathscr{F}_\vr '(w_\vr)v_\vr
+\int_0^1(1-s)
\mathscr{F}_\vr ''(w_\vr+sv_\vr )[v_\vr ,v_\vr ]
-\int P^0_\vr (x)g(\jdz{u_\vr}) u_\vr \cdot\overline{v_\vr }\,.
\end{aligned}
\]
So we deduce, noticing that $0\leq P^0_\vr (x)\leq m$,
\begin{equation}\label{v-eps estimate}
\begin{split}
 &\frac{1}{2}\norm{v_\vr }^2+\frac{\omega}{2}\jdz{v_\vr }_2^2+\int_0^1(1-s)
 \Psi_\vr ''(w_\vr+sv_\vr )[v_\vr ,v_\vr ]\\
 \leq&\jdz{\mathscr{F}_\vr '(w_\vr)v_\vr }
 +\int P^0_\vr (x)g(\jdz{u_\vr})
 \jdz{u_\vr}\cdot\jdz{v_\vr }.
\end{split}
\end{equation}
Since $t_\vr\to t_0$, it is clear that $\{u_\vr\}, \{w_\vr\}$ and $\{v_\vr\}$ are bounded,
hence, by the definitions and \eqref{estimates of Gamma-eps0},
\eqref{estimates of Gamma-eps1}, 
\[
\mathscr{F}_\vr(z_\vr)=o(1), \quad \|\mathscr{F}'_\vr(z_\vr)\|=o(1)
\]
as $\vr\to 0$ for $z_\vr=u_\vr, w_\vr, v_\vr$.
In addition, by noting that for $q\in[2,3]$
$$
\limsup_{r\to\infty}\int_{\jdz{x}>r}\jdz{u_\vr}^{q}=0,
$$
using the assumption $0\in\mathscr{P}$ one deduces
\[
\begin{aligned}
&\int \bkt{P^0_\vr (x)}^{q/(q-1)}\jdz{u_\vr}^{q}\\
=&\bigg(\int_{\jdz{x}\leq r}+\int_{\jdz{x}>r}\bigg){P^0_\vr (x)}^{q/(q-1)}\jdz{u_\vr}^{q}\\
\leq&\int_{\jdz{x}\leq r}\bkt{P^0_\vr (x)}^{q/(q-1)}\jdz{u_\vr}^{q}
+m^{q/(q-1)}\int_{\jdz{x}>r}\jdz{u_\vr}^{q}\\
=&\ o(1)
\end{aligned}
\]
as $\vr \to0$. Thus by (\ref{v-eps estimate}) one has
$\norm{v_\vr}^2\to0$, that is, $h_\vr
(t_\vr  e)\to\mathscr{J}_m(t_0 e)$. Consequently,
$$
\int P^0_\vr (x)G(\jdz{w_\vr})\to0
$$
as $\vr \to0$. This, jointly with (\ref{tilde-Phi-eps----Gm}), shows
\[
\widetilde{\Phi}_\vr (w_\vr)=\mathscr{T}_m(w_\vr)+o(1)
 =\mathscr{T}_m(u_\vr)+o(1),
\]
that is,
$$
I_\vr (t_\vr  e)=J_m(t_0e)+o(1)
$$
as $\vr \to0$. Then, since
$$
J_m(t_0e)\leq\max_{v\in E_e}\mathscr{T}_m(v)=J_m(e)=\ga_m,
$$
we obtain by using \eqref{d-eps>gamma} and
(\ref{d-eps<I-eps(t-eps)})
$$
\ga_m\leq\lim_{\vr \to0}c_\vr \leq\lim_{\vr \to0}I_\vr (t_\vr
e)=J_m(t_0e)\leq\ga_m,
$$
hence, $c_\vr \to\gamma_m$.
\end{proof}

\medskip

\begin{proof}[Proof of Lemma \ref{d-eps is attained}]
Given $\vr >0$, let $\{u_n\}\subset\mathscr{N}_\vr $ be a
minimizing sequence: $I_\vr (u_n)\to c_\vr $. By the Ekeland
variational principle we can assume that $\{u_n\}$ is in fact a
$(PS)_{c_\vr }$-sequence for $I_\vr $ on $E^+$ (see
\cite{Pankov,Willem}). Then $w_n=u_n+h_\vr (u_n)$ is a $(PS)_{c_\vr
}$-sequence for $\widetilde{\Phi}_\vr $ on $E$. It is clear that
$\{w_n\}$ is bounded, hence is a $(C)_{c_\vr}$-sequence. We can assume
without loss of generality that $w_n\rightharpoonup w_\vr =w_\vr
^++w_\vr ^-\in\mathscr{K}_\vr $ in $E$. If $w_\vr \not=0$ then
$\widetilde{\Phi}_\vr (w_\vr )=c_\vr $. So we are going to show that
$w_\vr \not=0$ for all small $\vr >0$.

For this end, take $\limsup_{\jdz{x}\to\infty}P(x)<\kappa<m$ and
define
$$
P^\kappa(x)=\min\{\kappa,P(x)\}.
$$
Consider the functional
$$
\widetilde{\Phi}_\vr
^\kappa(u)=\frac{1}{2}\bkt{\|u^+\|^2-\|u^-\|^2-\omega\jdz{u}_2^2}
-\mathscr{F}_\vr (u)-\int P^\kappa_\vr (x)G(\jdz{u})
$$
and as before define correspondingly $h_\vr ^\kappa:E^+\to E^-$,
$I_\vr ^\kappa:E^+\to\mathbb{R}$, $\mathscr{N}_\vr ^\kappa$, $c_\vr
^\kappa$ and so on. Following the proof of Lemma \ref{d to gamma}, one
finds
\begin{equation}\label{d-eps-sigma  to  gamma-sigma}
\lim_{\vr \to0}c_\vr ^\kappa=\gamma_\kappa.
\end{equation}

Assume by contradiction that there is a sequence $\vr _j\to0$ with
$w_{\vr _j}=0$. Then $w_n=u_n+h_{\vr _j}(u_n)\rightharpoonup0$ in
$E$, $u_n\to0$ in $L_{loc}^q$ for $q\in[1,3)$, and $w_n(x)\to0$ a.e.
in $x\in\mathbb{R}^3$. Let $t_n>0$ be such that
$t_nu_n\in\mathscr{N}_{\vr _j}^\kappa$. Since $u_n\in\mathscr{N}_\vr
$, it is not difficult to see that $\{t_n\}$ is bounded and one may
assume $t_n\to t_0$ as $n\to\infty$. By $(P_0)$, the set $A_\vr
:=\{x\in\mathbb{R}^3:P_\vr (x)>\kappa\}$ is bounded. Remark that
$h_{\vr _j}^\kappa(t_nu_n)\rightharpoonup0$ in $E$ and $h_{\vr
_j}^\kappa(t_nu_n)\to0$ in $L_{loc}^q$ for $q\in[1,3)$ as
$n\to\infty$ (see \cite{Ackermann}). Moreover, by virtue of Lemma \ref{ding2010},
$\widetilde{\Phi}_{\vr _j}(t_nu_n+h_{\vr _j}^\kappa(t_nu_n))\leq
I_{\vr _j}(u_n)$. We obtain
\[
\begin{aligned}
c_{\vr _j}^\kappa&\leq I_{\vr _j}^\kappa(t_nu_n)
=\widetilde{\Phi}_{\vr _j}^\kappa
(t_nu_n+h_{\vr _j}^\kappa(t_nu_n))\\
 &=\widetilde{\Phi}_{\vr _j}(t_nu_n+h_{\vr _j}^\kappa(t_nu_n))
 +\int \big(P_{\vr _j}(x)-P_{\vr _j}^\kappa(x)\big)
 G\big(|t_nu_n+h_{\vr _j}^\kappa(t_nu_n)|\big)\\
 &\leq I_{\vr _j}(u_n)+\int_{A_{\vr _j}}
 \big(P_{\vr _j}(x)-P_{\vr _j}^\kappa(x)\big)
 G\big(|t_nu_n+h_{\vr _j}^\kappa(t_nu_n)|\big)\\
 &=c_{\vr _j}+o(1)
\end{aligned}
\]
as $n\to\infty$. Hence $c_{\vr _j}^\kappa\leq c_{\vr _j}$. By
(\ref{d-eps-sigma  to  gamma-sigma}), letting $j\to\infty$ yields
$$
\gamma_\kappa\leq\gamma_m,
$$
which contradiction with $\gamma_m<\gamma_\kappa$.
\end{proof}

\medskip

\begin{proof}[Proof of Lemma \ref{least energy solution compact}]
Since $\mathscr{L}_\vr \subset B_\Lambda$ for all small $\vr >0$,
assume by contradiction that, for some $\vr _j\to0$,
$\mathscr{L}_{\vr _j}$ is not compact in $E$. Then we can choose
$u_n^j\in\mathscr{L}_{\vr _j}$ be such that $u_n^j\rightharpoonup0$ as
$n\to\infty$, as done for proving the Lemma \ref{d-eps is attained},
we gets a contradiction.

Now let $\{u_n\}\subset\mathscr{L}_\vr $ such that $u_n\to u$ in
$E$, and recall $H_0=i\alpha\cdot\nabla-a\beta$, by
$$
H_0u=\omega u+Q_\vr (x)A_{\vr ,u}^0u- \sum_{k=1}^3Q_\vr
(x)\alpha_kA_{\vr ,u}^ku+P_\vr (x)g(\jdz{u})u
$$
one has
\begin{equation}\label{H1 estimate}
\begin{split}
\jdz{H_0(u_n-u)}_2\leq&\ \omega\jdz{u_n-u}_2
 +\jdz{Q_\vr (x)\bkt{A_{\vr ,u_n}^0u_n-A_{\vr ,u}^0u}}_2\\
 &+\sum_{k=1}^3\jdz{Q_\vr (x)\alpha_k\big(A_{\vr ,u_n}^ku_n-A_{\vr ,u}^ku\big)}_2\\
 &+\jdz{P_\vr (x)\bkt{g(\jdz{u_n})u_n-g(\jdz{u})u}}_2
\end{split}
\end{equation}
A standard calculus shows that
\[
\begin{aligned}
\jdz{Q_\vr \cdot\alpha_k\big(A_{\vr ,u_n}^ku_n -A_{\vr
,u}^ku\big)}_2&\leq\jdz{Q}_\infty\jdz{u_n}_\infty^{1/6}
\jdz{A_{\vr ,u_n}^k-A_{\vr ,u}^k}_6\jdz{u_n}_{5/2}^{5/6}\\
 &+\jdz{Q}_\infty\jdz{u_n-u}_\infty^{1/6}\jdz{A_{\vr ,u}^k}_6\jdz{u_n-u}_{5/2}^{5/6}
\end{aligned}
\]
and
\[
\begin{aligned}
&\jdz{P_\vr \cdot\bkt{g(\jdz{u_n})u_n-g(\jdz{u})u}}_2\\
\leq&\jdz{P}_\infty\jdz{g(\jdz{u_n})-g(\jdz{u})}_\infty^{\frac{1}{2}}
\jdz{(g(\jdz{u_n})-g(\jdz{u}))^{1/2}u_n}_2\\
 &+\jdz{P}_\infty\jdz{g(\jdz{u})}_\infty\jdz{u_n-u}_2.
\end{aligned}
\]
By Lemma \ref{lemma of Aj} and the fact that $u_n\to u$ in
$L^q(\mathbb{R}^3,\mathbb{C}^4)$ for all $q\in[2,3]$, one gets
$\jdz{H_0(u_n-u)}_2\to0$, so $u_n\to u$ in
$H^1(\mathbb{R}^3,\mathbb{C}^4)$. With Lemma
\ref{critical points in W1,q}, $u_n\to u$ in
$W^{1,q}$ for all $q\in[2,\infty)$.
\end{proof}

\subsection*{Part 2. Concentration}
The proof relies on the following lemma. To prove it, it
suffices to show that for any sequence $\vr_j\to0$ the corresponding
sequence of solutions $u_j\in\mathscr{L}_{\vr_j}$ converges, up to a
shift of $x$-variable, to a least energy solution of the limit
problem \eqref{the limit problem}.

\begin{Lem}\label{concentration}
There is a maximum point $x_\vr $ of $\jdz{u_\vr }$ such that
$\mathrm{dist}(y_\vr ,\mathscr{P})\to0$ where $y_\vr =\vr  x_\vr $,
and for any such $x_\vr $, $v_\vr (x):=u_\vr (x+x_\vr )$ converges
to a least energy solution of (\ref{the limit problem}) in $W^{1,q}$
as $\vr \to0$ for all $q\geq2$.
\end{Lem}

\begin{proof}
Let $\vr _j\to0$, $u_j\in\mathscr{L}_j$, where
$\mathscr{L}_j=\mathscr{L}_{\vr _j}$. Then $\{u_j\}$ is bounded. A
standard concentration argument (see \cite{Lions}) shows that there
exist a sequence $\{x_j\}\subset\mathbb{R}^3$ and constant $R>0$,
$\delta>0$ such that
$$\liminf_{j\to\infty}\int_{B(x_j,R)}\jdz{u_j}^2\geq\delta.$$
Set
$$v_j=u_j(x+x_j).$$
Then $v_j$ solves, denoting
$\hat{Q}_j(x)=Q(\vr _j(x+x_j))$,
$\hat{A}_{\vr ,u_j}^k(x)={A}_{\vr ,u_j}^k(x+x_j)$ and
$\hat{P}_j(x)=P(\vr _j(x+x_j))$,
\begin{equation}\label{vj equa}
H_0 v_j-\omega v_j-\hat{Q}_j\hat{A}_{\vr ,u_j}^0v_j
+\sum_{k=1}^3\hat{Q}_j\alpha_k\hat{A}_{\vr ,u_j}^kv_j=\hat{P}_j\cdot g(\jdz{v_j})v_j,
\end{equation}
with energy
\[
\begin{aligned}
S(v_j)&:=\frac{1}{2}\big(\|v_j^+\|^2-\|v_j^-\|^2-\omega
|v_j|_2^2\big)
-\hat{\Gamma}_j(v_j)-\int\hat{P}_j(x)G(\jdz{v_j})\\
 &=\widetilde{\Phi}_j(v_j)={\Phi}_j(v_j)=\hat{\Gamma}_j(v_j)
 +\int\hat{P}_j(x)\widehat{G}(\jdz{v_j})\\
 &=c_{\vr _j}.
\end{aligned}
\]
Additionally, $v_j\rightharpoonup v$ in $E$ and $v_j\to v$ in
$L_{loc}^q$ for $q\in[1,3)$.

We now turn to prove that $\{\vr _jx_j\}$ is bounded. Arguing
indirectly we assume $\vr _j\jdz{x_j}\to\infty$ and get a
contradiction.

Without loss of generality assume $P(\vr _jx_j)\to P_\infty$.
Clearly, $m>P_\infty$ by $(P_0)$. Since for any $\psi\in C_c^\infty$
\[
\begin{aligned}
0&=\lim_{j\to\infty}\int\bigg(H_0v_j-\omega
v_j-\hat{Q}_j\hat{A}_{\vr ,u_j}^0v_j
+\sum_{k=1}^3\hat{Q}_j\alpha_k\hat{A}_{\vr ,u_j}^kv_j
-\hat{P}_jg(\jdz{v_j})v_j\bigg)\bar{\psi}\\
 &=\lim_{j\to\infty}\int \bkt{H_0v-\omega v-{P}_\infty g(\jdz{v})v}\bar{\psi},
\end{aligned}
\]
hence $v$ solves
$$
i\alpha\cdot\nabla v-a\beta v-\omega v={P}_\infty g(\jdz{v})v.
$$
Therefore,
$$
S_\infty(v):=\frac{1}{2}\bkt{\|v^+\|^2-\|v^-\|^2-\omega |v|_2^2}
-\int{P}_\infty G(\jdz{v})\geq\ga_{P_\infty}.
$$
It follows from $m>P_\infty$, by
Lemma \ref{ding2010}, one has $\ga_m<\ga_{P_\infty}$.
Moreover, by Fatou's lemma,
\[
\begin{aligned}
&\lim_{j\to\infty}\int\hat{P}_j(x)\widehat G (\jdz{v_j})
\geq \int{P}_\infty\widehat G (\jdz{v})
=S_\infty(v).
\end{aligned}
\]
Consequently, noting that $\hat{\Gamma}_j(v_j)=o(1)$ as $j\to\infty$,
$$
\gamma_m<\gamma_{P_\infty}\leq S_\infty(v)\leq\lim_{j\to\infty}c_{\vr _j}=\gamma_m,
$$
a contradiction.

Thus $\{\vr _jx_j\}$ is bounded. Hence, we can assume
$y_j=\vr _jx_j\to y_0$. Then $v$ solves
\begin{equation}\label{v solves equa}
i\alpha\cdot\nabla v-a\beta v-\omega v={P}(y_0)g(\jdz{v})v.
\end{equation}
Since $P(y_0)\leq m$, we obtain
$$
S_0(v):=\frac{1}{2}\bkt{\|v^+\|^2-\|v^-\|^2-\omega
|v|_2^2}-\int{P}(y_0)G(\jdz{v}) \geq\gamma_{P(y_0)}\geq\gamma_m.
$$
Again, by Fatou's lemma, we have
$$
S_0(v)=\int{P}(y_0)\widehat G (\jdz{v})
\leq\lim_{j\to\infty}c_{\vr _j}=\gamma_m.
$$
Therefore, $\gamma_{P(y_0)}=\gamma_m$, which implies
$y_0\in\mathscr{P}$ by Lemma \ref{ding2010}. By virtue of Lemma
\ref{critical points in W1,q} and (\ref{C-infty estimate of L}) it
is clear that one may assume that $x_j\in\mathbb{R}^3$ is a maximum
point of $\jdz{u_j}$. Moreover, from the above argument we readily
see that any sequence of such points satisfies
$y_j=\vr _jx_j$, converging to some point in $\mathscr{P}$ as
$j\to\infty$.

In order to prove $v_j\to v$ in $E$, recall that as the argument shows
$$\lim_{j\to\infty}\int\hat{P}_j(x)\widehat G (\jdz{v_j})
=\int{P}(y_0)\widehat G (\jdz{v}).$$
By $(g_2)$ and the exponential decay of $v$, using the Brezis-Lieb lemma,
one obtains $\jdz{v_j-v}_\sigma\to0$, then
$|v_j^\pm-v^\pm|_\sigma\to0$ by (\ref{lpdec}). Denote $z_j=v_j-v$.
Remark that $\{z_j\}$ is bounded in $E$ and $z_j\to0$ in
$L^\sigma$, therefore $z_j\to0$ in $L^q$ for all $q\in(2,3)$. The scalar
product of (\ref{vj equa}) with $z_j^+$ yields
$$
\big\langle v_j^+, z_j^+\big\rangle=o(1).
$$
Similarly, using the exponential decay of $v$ together with the fact
that $z_j^\pm\to0$ in $L_{loc}^q$ for $q\in[1,3)$, it follows from
(\ref{v solves equa}) that
$$
\big\langle v^+, z_j^+\big\rangle=o(1).
$$
Thus
$$
\|z_j^+\|=o(1),
$$
and the same arguments show
$$
\|z_j^-\|=o(1),
$$
we then get $v_j\to v$ in $E$, and the arguments in Lemma
\ref{least energy solution compact} shows that $v_j\to v$ in
$W^{1,q}$ for all $q\geq2$.
\end{proof}

\subsection*{Part 3. Exponential decay}
See the following Lemma \ref{exp decay}.
For the later use denote $D=i\alpha\cdot\nabla$ and, for
$u\in\mathscr{L}_\vr $,  write (\ref{D22}) as
$$
Du=a\beta u+\omega u+Q_\vr (x)A_{\vr ,u}^0u- \sum_{k=1}^3Q_\vr
(x)\alpha_kA_{\vr ,u}^ku+P_\vr (x)g(\jdz{u})u.
$$
Applying the operator $D$ on both sides and noting that
$D^2=-\Delta$, we get
\begin{equation}\label{elliptic equ}
\begin{split}
\Delta u=&\ a^2u-\bkt{\omega+Q_\vr (x)A_{\vr ,u}^0(x)+P_\vr (x)g(\jdz{u})}^2u\\
 &-D\bkt{P_\vr (x)g(\jdz{u})}u-D\bkt{Q_\vr  A_{\vr ,u}^0}u\\
 &+\sum_{k=1}^3\bkt{Q_\vr  A_{\vr ,u}^k}^2u+\sum_{k=1}^3D\bkt{Q_\vr  A_{\vr ,u}^k}\alpha_ku\\
 &+2i\sum_{k=1}^3Q_\vr  A_{\vr ,u}^k\partial_ku.
\end{split}
\end{equation}
With the fact that
\begin{equation}\label{identity}
\Delta\jdz{u}^2=\bar{u}\Delta u+u\Delta\bar{u}+2\jdz{\nabla u}^2
\end{equation}
and $\overline{\alpha_ku}\cdot u=\alpha_ku\cdot\bar{u}$ one deduces
\[
\begin{aligned}
\Delta\jdz{u}^2=&\ 2a^2\jdz{u}^2-2\bkt{\omega+Q_\vr (x)A_{\vr ,u}^0(x)
+P_\vr (x)g(\jdz{u})}^2\jdz{u}^2     \\
 &+2\sum_{k=1}^3\bkt{Q_\vr  A_{\vr ,u}^k}^2\jdz{u}^2+2i\sum_{k=1}^3
 \sum_{1\leq j\leq3 \atop j\not
 =k}\partial_j\bkt{Q_\vr  A_{\vr ,u}^k}(\al_j\al_ku)\cdot\bar{u}\\
 &+4\Im{\sum_{k=1}^3Q_\vr  A_{\vr ,u}^k\partial_ku\cdot\bar{u}}+2\jdz{\nabla u}^2.
\end{aligned}
\]
In addition, setting
$$
f^0_\vr (x):=\max\hbt{\jdz{Q_\vr (x)A_{\vr ,u}^k(x)}:k=0,1,2,3},
$$
$$
f^1_\vr (x):=\max\hbt{\jdz{\nabla\big(Q_\vr (x)A_{\vr
,u}^k(x)\big)}:k=0,1,2,3},
$$
one has
\[
\bigg|2i\sum_{k=1}^3\sum_{1\leq j\leq3 \atop
j\not=k}\partial_j\bkt{Q_\vr A_{\vr
,u}^k}(\al_j\al_ku)\cdot\bar{u}\bigg|\leq c_1f_\vr
^1(x)\jdz{u}^2
\]
and
\[
\bigg|4\Im{\sum_{k=1}^3Q_\vr  A_{\vr
,u}^k\partial_ku\cdot\bar{u}}\bigg|\leq c_2f_\vr
^0(x)\bkt{\jdz{\nabla u}^2+\jdz{u}^2}.
\]
Hence
\[
\begin{aligned}
\Delta\jdz{u}^2\geq&\bigg(2a^2-2\bkt{\omega+Q_\vr  A_{\vr ,u}^0+P_\vr
(x)g(\jdz{u})}^2+
2\sum_{k=1}^3\bkt{Q_\vr  A_{\vr ,u}^k}^2\bigg)\jdz{u}^2\\
 &-c_1f_\vr ^1(x)\jdz{u}^2-c_2f_\vr ^0(x)
 \bkt{\jdz{\nabla u}^2+\jdz{u}^2}+2\jdz{\nabla u}^2.
\end{aligned}
\]
Observe that for $\vr >0$ sufficiently small, by (\ref{Ak local estimates}), we
get
$$
c_2\jdz{f^0_\vr }<2,
$$
hence
\begin{equation}\label{Delta u2 inequ}
\begin{split}
\Delta\jdz{u}^2\geq&2\bigg(a^2-\bkt{\omega+Q_\vr  A_{\vr ,u}^0+P_\vr
\cdot g(\jdz{u})}^2
+\sum_{k=1}^3\bkt{Q_\vr  A_{\vr ,u}^k}^2\bigg)\jdz{u}^2\\
 &-c_1f_\vr ^1(x)\jdz{u}^2-c_2f_\vr ^0(x){\jdz{u}^2}.
\end{split}
\end{equation}
This together with the regularity results for $u$
implies there is $M>0$ satisfying
$$
\Delta\jdz{u}^2\geq -M\jdz{u}^2.
$$
By the sub-solution estimate \cite{Trudinger,Simon}, one has
\begin{equation}\label{jifen guji}
\jdz{u(x)}\leq C_0\bigg(\int_{B_1(x)}\jdz{u(y)}^2dy\bigg)^{1/2}
\end{equation}
with $C_0$ independent of $x$ and $u\in\mathscr{L}_\vr $, $\vr >0$
small.

\begin{Lem}\label{v-eps uniformly to 0}
Let $v_\vr $ and $\hat{A}_{\vr ,u_\vr }^k$ for $k=0,1,2,3$ be given in the
proof of Lemma \ref{concentration}. Then $\jdz{v_\vr (x)}\to0$ and
$\jdz{\hat{A}_{\vr ,u_\vr }^k(x)}\to0$ as $\jdz{x}\to\infty$ uniformly in
$\vr >0$ small.
\end{Lem}
\begin{proof}
Arguing indirectly, if the conclusion of the lemma is not held, then
by (\ref{jifen guji}), there exist $\delta>0$ and
$x_j\in\mathbb{R}^3$ with $\jdz{x_j}\to\infty$ such that
$$
\delta\leq\jdz{v_j(x_j)}\leq
C_0\Big(\int_{B_1(x_j)}\jdz{v_j}^2\Big)^{1/2},
$$
where $\vr _j\to0$ and $v_j=v_{\vr _j}$. Since $v_j\to v$ in $E$, we
obtain
\[
\begin{aligned}
\delta&\leq C_0\Big(\int_{B_1(x_j)}\jdz{v_j}^2\Big)^{1/2}\\
 &\leq C_0\Big(\int \jdz{v_j-v}^2\Big)^{1/2}+C_0\Big(\int_{B_1(x_j)}\jdz{v}^2\Big)^{1/2}\to0,
\end{aligned}
\]
a contradiction. Now, jointly with (\ref{Ak local estimates}), one
sees also $\jdz{\hat{A}_{\vr ,u_\vr }^k(x)}\to0$ as $\jdz{x}\to\infty$
uniformly in $\vr >0$ small.
\end{proof}

\begin{Lem}\label{exp decay}
There exist $C>0$ such that for all $\vr >0$ small
$$\jdz{u_\vr (x)}\leq Ce^{-\frac{c_0}{2}\jdz{x-x_\vr }}$$
where $c_0={\sqrt{\bkt{a^2-\omega^2}}}$.
\end{Lem}

\begin{proof}
The conclusions of Lemma \ref{v-eps uniformly to 0} with (\ref{Delta
u2 inequ}) allow us to take $R>0$ sufficiently large such that
\[
\Delta\jdz{v_\vr }^2\geq\bkt{a^2-\omega^2}\jdz{v_\vr }^2
\]
for all $\jdz{x}\geq R$ and $\vr >0$ small. Let
$\Gamma(y)=\Gamma(y,0)$ be a fundamental solution to
$-\Delta+\bkt{a^2-\omega^2}$. Using the uniform boundedness, we may
choose that $\jdz{v_\vr (y)}^2\leq\bkt{a^2-\omega^2}\Gamma(y)$ holds
on $\jdz{y}=R$ for all $\vr >0$ small. Let $z_\vr =\jdz{v_\vr
}^2-\bkt{a^2-\omega^2}\Gamma$. Then
\[
\begin{aligned}
\Delta z_\vr &=\Delta\jdz{v_\vr }^2-\bkt{a^2-\omega^2}\Delta\Gamma\\
 &\geq\bkt{a^2-\omega^2}\bkt{\jdz{v_\vr }^2-\bkt{a^2-\omega^2}\Gamma}
 =\bkt{a^2-\omega^2}z_\vr .
\end{aligned}
\]
By the maximum principle we can conclude that $z_\vr (y)\leq0$ on
$\jdz{y}\geq R$. It is well known that there is $C'>0$ such that
$\Gamma(y)\leq C'\exp(-c_0\jdz{y})$ on $\jdz{y}\geq1$, we see that
$$
\jdz{v_\vr (y)}^2\leq C''e^{-c_0\jdz{y}}
$$
for all $y\in\mathbb{R}^3$ and all $\vr >0$ small, that is
$$
\jdz{u_\vr (x)}\leq Ce^{-\frac{c_0}{2}\jdz{x-x_\vr }}
$$
as claimed.
\end{proof}

Now, with the above arguments, we are ready to prove Theorem
\ref{main theorem}.

\begin{proof}[Proof of Theorem \ref{main theorem}]
Going back to system (\ref{D2}) with the variable substitution:
$x\mapsto x/\vr $, Lemma \ref{d-eps is attained} jointly with Lemma
\ref{critical points in W1,q}, shows that, for all $\vr >0$ small,
Eq.(\ref{D2}) has at least one least energy solution $w_\vr \in
W^{1,q}$ for all $q\geq2$. In addition, if $P,\ Q\in
C^{1,1}(\mathbb{R}^3)$, with (\ref{elliptic equ}) and the elliptic
regularity (see \cite{Trudinger}) one obtains a classical
solution, that is, the conclusion (i) of Theorem \ref{main theorem}.
And Lemma \ref{least energy solution compact} is nothing but the
conclusion (ii). Finally, the conclusion (iii) and (iv) follow from
Lemma \ref{concentration} and Lemma \ref{exp decay} respectively.
\end{proof}

\medskip

\noindent {\it Acknowledgment.}

Special thanks to the reviewer for his/her good comments and suggestions. The comments were
all valuable and helpful for revising and improving our paper,
as well as to the important guiding
significance of our research.

The work was supported by the
National Science Foundation of China (NSFC10831005, 10721061,
11171286).

\end{document}